\documentclass{article}
\usepackage[utf8]{inputenc}
\usepackage{amssymb, amsmath, amsthm, fancyhdr, geometry, tikz, multicol, hyperref, esint, float}
\usetikzlibrary{shapes}
\usepackage[mathscr]{eucal} % Zapf's Euler script
\usepackage{enumitem} % letters with [label=\alph*]
\renewcommand{\;}{\:\:\:} % more space
\newcommand{\f}[2]{\frac{#1}{#2}} % fraction
\newcommand{\R}{\mathbb{R}} % real numbers
\renewcommand{\tilde}[1]{\widetilde{#1}} % tilde
\newcommand{\eps}{\epsilon} % epsilon
\newcommand{\p}{\partial} % del
\renewcommand{\*}{\cdot} % multiplication
\renewcommand{\,}{,\:} % space after a comma
\usepackage{blkarray} % labeled arrays
\renewcommand{\Re}{\operatorname{Re}} % textual real part
\renewcommand{\Im}{\operatorname{Im}} % textual imaginary part
\numberwithin{equation}{section} % equation numbers eg 2.1
\usetikzlibrary{decorations.pathreplacing, calligraphy}

\makeatletter
\newcommand{\address}[1]{\gdef\@address{#1}}
\newcommand{\email}[1]{\gdef\@email{\url{#1}}}
\newcommand{\@endstuff}{\par\vspace{\baselineskip}\noindent\small
\begin{tabular}{@{}l}\scshape\@address\\\textit{E-mail address:} \@email\end{tabular}}
\AtEndDocument{\@endstuff}
\makeatother

\usepackage[backend=biber, sorting=nyt, style=numeric]{biblatex}
\usepackage{amsthm}
\theoremstyle{plain}
\newtheorem{theorem}{Theorem}
\newtheorem{conjecture}{Conjecture}
\newtheorem{corollary}[theorem]{Corollary}
\theoremstyle{definition}

\newtheorem{lemma}{Lemma}[section]
\newtheorem{definition}{Definition}
\newtheorem{assumption}{Assumption}
\newtheorem{proposition}{Proposition}[section]
\newtheorem{question}{Question}
\theoremstyle{remark}
\newtheorem{example}{Example}

\title{Explicit Semiclassical Resonances from Many Delta Functions}
\author{Ethan J. Brady}
\date{July 2024} % prepared January 2025
\address{Division of Applied Mathematics, Brown University, Providence, RI 02906}
\email{ethan_brady@brown.edu}

\begin{document}

\maketitle

\begin{abstract}
    We study the scattering resonances arising from multiple $h$-dependent Dirac delta functions on the real line in the semiclassical regime $h \rightarrow 0$. We focus on resonances lying in strings along curves of the form $\Im z \sim -\gamma h\log(1/h)$ and find that resonances along such strings exist if and only if $\gamma$ is a slope of a Newton polygon we construct from the parameters. Furthermore, the set of these $\gamma$ corresponds to a complete and disjoint partitioning of a line segment with delta functions at interval endpoints. Hence, there are at most $N-1$ strings of resonances from $N$ delta functions, improving a bound from (Datchev, Marzuola, \& Wunsch 2023). Lastly, we identify a `dominant pair' of delta functions in the sense that they correspond to the longest-living string of resonances, this string is the only one of logarithmic shape with respect to $\Re z$, and no delta functions between them can contribute to strings of resonances. The simple properties of delta functions permit elementary proofs, and we provide many visual examples to demonstrate the results.
\end{abstract}

\tableofcontents

\section{Introduction}

We investigate the resonances of a semiclassical quantum system in one dimension. Resonances $z \in \mathbb{C}$ generalize bound state eigenvalues to potential functions where energy eventually dissipates at infinity. The real and imaginary part of $z$ respectively give the oscillation and decay of waves $u(x,t)$ \cite{dyatlov_zworski}. The semiclassical approximation $h\rightarrow 0$ is appropriate for large or high energy particles. A general goal is to find how the geometry of a potential function $V(x)$ determines the resonances.
\begin{align}
    (-h^2 \f{\p^2}{\p x^2} + V - z^2) u = 0 \label{resonantState}
\end{align}
Specifically, we consider scattering off multiple thin barriers in one dimension, modeled by Dirac delta functions:
\begin{align}
    V(x) := \sum_{j=1}^N C_j h^{1+\beta_j} \delta(x-x_j) \label{V}
\end{align}
The $h$-dependent strength coefficient $C_j h^{1+\beta_j}$ is appropriate for a frequency-dependent interaction with the boundary \cite{Barr_et_al}. The parameters $C_j$, $\beta_j$, and $x_j$ are real with $C_j \neq 0$ and $x_1 < ... < x_N$. We restrict parameters to $\beta_j > 0$, a regime where transmission dominates reflection \cite{Datchev_Malawo}, and impose an additional assumption \ref{assumption} discussed later. One-dimensional systems with one or two Dirac deltas are often a pedagogical example \cite{Griffiths, Belloni_Robinett}, but we find this natural generalization to $N \geq 2$ delta functions reveals novel and explicit behavior. Figure \ref{fig intro setup} depicts the setup of (eq \ref{V}), shows the resonances for a case of $N=4$, and demonstrates the agreement of numerics and our results.
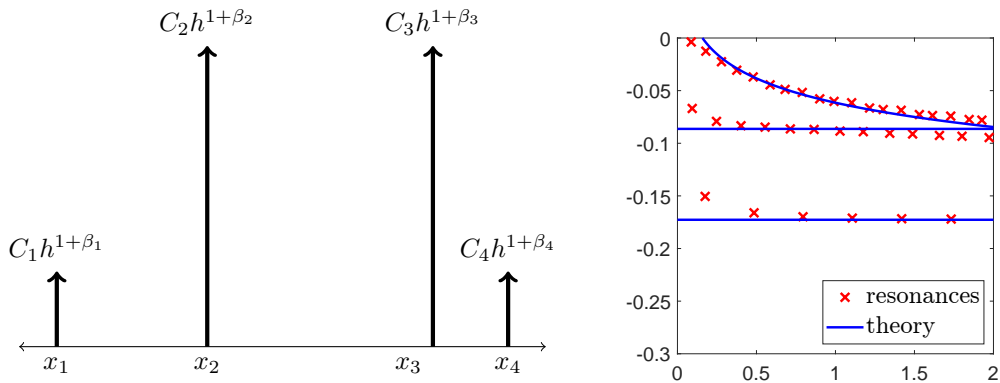
\begin{figure}[H]
\centering
\resizebox{0.85\textwidth}{!}{
\begin{tikzpicture} % deltas
[scale=1.0, ultra thick]
    \draw[->] (0,0) -- (0,1);
    \draw[->] (2,0) -- (2,4);
    \draw[->] (5,0) -- (5,4);
    \draw[->] (6,0) -- (6,1);
    \node[anchor=south] (b1) at (0,1) {$C_1 h^{1+\beta_1}$};
    \node[anchor=south] (b2) at (2,4) {$C_2 h^{1+\beta_2}$};
    \node[anchor=south] (b3) at (5,4) {$C_3 h^{1+\beta_3}$};
    \node[anchor=south] (b4) at (6,1) {$C_4 h^{1+\beta_4}$};
    \node[anchor=north] (b1) at (0,0) {$x_1$};
    \node[anchor=north] (b2) at (2,0) {$x_2$};
    \node[anchor=north east] (b3) at (5,0) {$x_3$};
    \node[anchor=north] (b4) at (6,0) {$x_4$};
    \draw[<->, thin] (-0.5,0) -- (6.5,0);
\end{tikzpicture} \qquad 
\includegraphics[width=0.35\textwidth]{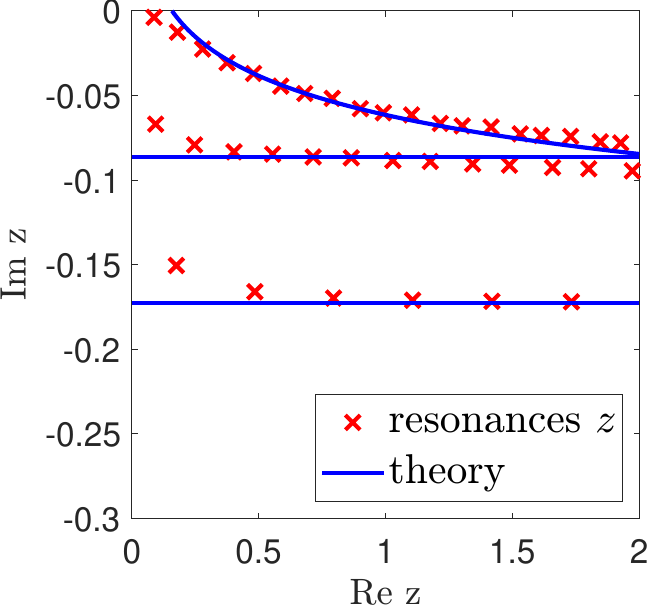}
} % resize box
\caption{General setup of the potential function and resulting resonances. Upward arrows represent delta functions against the horizontal line $x \in \R$. The MATLAB plot demonstrates the presence of multiple strings of resonances, for the parameters given in Example \ref{ex N4} case $c$.}
\label{fig intro setup}
\end{figure}

This paper improves the results of \cite{Datchev_Marzuola_Wunsch} for the problem described above. They found that resonances lie in strings on multiple logarithmic curves of the form $\Im z \sim - \gamma h\log(1/h)$ and specified $2^{N-1}-1$ possibilities for $\gamma$ determined by the parameters $x_j$, $\beta_j$. Here, we use the same tool of Newton polygons but a more precise calculation to find exactly which $\gamma$ occur. We improve the upper bound to $N-1$ and show it is sharp (Corollary \ref{cor number strings}), as well as provide a new interpretation of these strings (Theorem \ref{thm Gamma}). We also show the existence of a dominant pair of delta functions governing the resonances closest to the real axis and prove all strings except this first one are flat up to $o(h)$.

\par Resonances lying on multiple logarithmic strings of the form $\Im z \sim -\gamma h\log(1/h)$ have been observed in various examples but remain poorly understood. These include delta functions on the half line \cite{Datchev_Malawo} and in $\R^d$ \cite{Galkowski}; singular, compactly supported potentials on the real line \cite{Zworski}; two obstacles in $\R^2$, one with a corner \cite{Burq}; and manifolds with conic singularities \cite{Hillairet_Wunsch}. At best, the equation and quantity of all possible $\gamma$ can be determined explicitly, but other examples have just one possible $\gamma$ known and/or resonance-free regions. This paper provides a simple but well understood example to suggest approaches for more complicated problems. One such consistent pattern is the length of some longest geodesic in the denominator of $\gamma$, most explicitly in \cite{Hillairet_Wunsch} as the longest trajectory between cone points. We observe a similar form but with a dependence on the strength of singularity like in \cite{Zworski}.

Another motivation of our work is to a greater understanding of delta functions. Their application to quantum corrals and leaky quantum graphs have been studied extensively in \cite{Galkowski}. In one dimension, multiple delta functions have been used to approximate general smooth potentials \cite{Sahu_Sahu}. Other recent studies of $N>2$ delta functions in one dimension include \cite{Tanimu, Erman_et_al_2018, Erman_et_al_2017}, and we discuss speculative connection to these in Question \ref{q bound states}. However, we are not aware of any in the semiclassical regime besides \cite{Datchev_Marzuola_Wunsch} and \cite{Datchev_Malawo}.

\section{Results}

We must construct a particular Newton polygon before presenting the main theorems. We use Examples \ref{ex equal strength} - \ref{ex N4} to explain the theorems, then discuss conjectures and questions. We conclude this section with an outline of the proof, covering most of the remaining sections. However, we emphasize here that Section \ref{sec:numerics}, besides explaining the numerical method, uses  Examples \ref{ex exact resonances} - \ref{ex C h} to critique the agreement of numerics and theory, and to explain subtle gaps regarding assumptions and future work. For now, we summarize and interpret the results below.

\begin{itemize}
    \item Resonances lie in strings of the form $\Im z \sim -\gamma h\log(1/h)$ with $\gamma$ determined in an indirect but explicit way from the parameters (Theorem \ref{thm structure}). We prove these resonances exist (Theorem \ref{thm existence}) and provide numerical evidence that such resonances are the only ones that exist.
    \item The strings of resonances correspond to a complete and disjoint partitioning of a line segment with Dirac deltas at interval endpoints. More precisely, each $\gamma$ incorporates parameters from precisely two Dirac deltas, and the intervals between these pairs of deltas form the partition (Theorem \ref{thm Gamma}).
    \item The string closest to the real axis, hence containing the longest living resonances, is the only string with logarithmic shape (Corollary \ref{cor log Re z}). Its corresponding pair of Dirac deltas minimize their combined strength divided by the distance between them, among all pairs.
\end{itemize}

\subsection{Newton polygon and theorems}

\begin{definition} \label{defn polygon}
    The \textit{Newton polygon} of a set of points $\{(\lambda_i, \nu_i)\}_i$ is the convex hull of the epigraph of the points. In other words, it is the shape of a rubber band around these points and stretched to $(0,\infty)$.
\end{definition}

Newton polygons are a tool from commutative algebra for investigating valuations of polynomial roots. However, we use them to encapsulate several minimization conditions and inequalities. As typical, the important feature is the non-infinite slopes.

\begin{definition} \label{defn our polygon}
    \textit{Our Newton polygon} is the one constructed from the following set of points: 
    \begin{align} \label{eq polygon points}
        P := \big\{\big(2(x_k-x_j)\, \beta_j+\beta_k\big) : 1\leq j < k \leq N\big\} \cup \big\{(0,0)\big\}
    \end{align}
    Let $\Gamma$ be the set of its non-infinite slopes:
    \begin{align}
        \Gamma := \Big\{\f{\nu - \tilde{\nu}}{\lambda - \tilde{\lambda}} : (\lambda, \nu) \text{ and } (\tilde{\lambda}, \tilde{\nu}) \in P\, \lambda > \tilde{\lambda} \Big\}
    \end{align}
\end{definition}
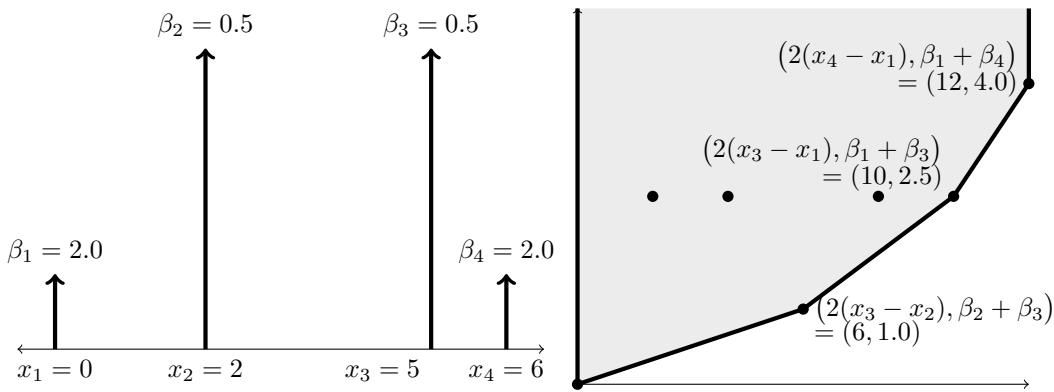
\begin{figure}
    \centering
    \resizebox{0.85\textwidth}{!}{
\begin{tikzpicture} % deltas
[scale=1.0, ultra thick]
    \draw[->] (0,0) -- (0,1);
    \draw[->] (2,0) -- (2,4);
    \draw[->] (5,0) -- (5,4);
    \draw[->] (6,0) -- (6,1);
    \node[anchor=south] (b1) at (0,1) {$\beta_1 = 2.0$};
    \node[anchor=south] (b2) at (2,4) {$\beta_2 = 0.5$};
    \node[anchor=south] (b3) at (5,4) {$\beta_3 = 0.5$};
    \node[anchor=south] (b4) at (6,1) {$\beta_4 = 2.0$};
    \node[anchor=north] (b1) at (0,0) {$x_1 = 0$};
    \node[anchor=north] (b2) at (2,0) {$x_2 = 2$};
    \node[anchor=north east] (b3) at (5,0) {$x_3 = 5$};
    \node[anchor=north] (b4) at (6,0) {$x_4 = 6$};
    \draw[<->, thin] (-0.5,0) -- (6.5,0);
\end{tikzpicture}
\begin{tikzpicture} % Newton polygon
    [scale=1.0, ultra thick]
    \draw[fill=gray!15] (0,5) -- (0,0) -- (3,1) -- (5,2.5) -- (6,4) -- (6,5);
    \draw[->, thin] (0,0) -- (0,5);
    \draw[->, thin] (0,0) -- (6,0);
    \node[anchor=west] (bT) at (3,1) {$\big(2(x_3-x_2), \beta_2+\beta_3\big)$};
    \node[anchor=north west] (bT) at (3,1) {$= (6,1.0)$};
    \node[anchor=south east] (cT) at (5,2.75) {$\big(2(x_3-x_1), \beta_1+\beta_3\big)$};
    \node[anchor=east] (cT) at (5,2.75) {$= (10, 2.5)$};
    \node[anchor=south east] (dT) at (6,4.0) {$\big(2(x_4-x_1), \beta_1+\beta_4\big)$};
    \node[anchor=east] (dT) at (6,4.0) {$= (12,4.0)$};
    \node[circle, fill, inner sep=1.5pt] (point) at (0,0) {};
    \node[circle, fill, inner sep=1.5pt] (point) at (1,2.5) {};
    \node[circle, fill, inner sep=1.5pt] (point) at (2,2.5) {};
    \node[circle, fill, inner sep=1.5pt] (point) at (3,1.0) {};
    \node[circle, fill, inner sep=1.5pt] (point) at (4,2.5) {};
    \node[circle, fill, inner sep=1.5pt] (point) at (5,2.5) {};
    \node[circle, fill, inner sep=1.5pt] (point) at (6,4.0) {};
    \node[anchor=north] (axes) at (3,0) {$\lambda = 2(x_k - x_j)$};
    \node[anchor=south, rotate=90] (axes) at (0,2.5) {$\nu = \beta_j + \beta_k$};
\end{tikzpicture}
} % resize box
    \caption{Example of our chosen Newton polygon from the parameters on left, labeled on each delta function (represented by arrows). Note that larger $\beta_j$ implies weaker delta functions, since $h\rightarrow 0$ in $C_j h^{\beta_j+1}$. We continue displaying Example \ref{ex N4} part $c$ as a case with several strings of resonances.}
    \label{fig intro polygon}
\end{figure}

Before our results, we require an additional assumption on the parameters introduced in (eq \ref{V}), which we state using the Newton polygon.

\begin{assumption} \label{assumption}
Let the parameters $\beta_j, x_j$ be given such that the non-infinite slopes of the Newton polygon are unique. In other words, for any three collinear points in (eq \ref{eq polygon points}), at least one is in the interior of the polygon.
\end{assumption}

This assumption is not always necessary but is convenient throughout. Example \ref{ex assumption numeric} shows that without this assumption, strings of higher multiplicities emerge. Fortunately, this case is rare, in the sense that the parameters following this assumption are generic in the parameter space, since we need only avoid a finite set of linear equations. We can now present a first result, which is Lemma 6 of \cite{Datchev_Marzuola_Wunsch} but reformulated for our subsequent use here.

\begin{theorem} \label{thm structure}
    Resonances on log curves $h\log(1/h)$ must correspond to slopes of our Newton polygon. More precisely, if $z = z(h) = z(h_j)$ is a sequence of resonances with $\Re z \in U \Subset (0,\infty)$ and $\Im z = -\gamma h\log(1/h) + O(h)$ as $h_j \rightarrow 0$, then $\gamma \in \Gamma$.
\end{theorem}

We define a string of resonances as the sequence $z(h_j)$ corresponding to a value $\gamma$, which in turn comes from the Newton polygon. Hence, describing strings of resonances reduces to describing slopes on the Newton polygon. Our next notation gives the structure of these slopes, with the convention $j < k$.
\begin{align} \label{eq gamma defn} 
    \gamma_{jk} &:= \f{\beta_k + \beta_j}{2(x_k - x_j)} &
    \tilde{\gamma}_{jk} &:= \f{|\beta_k-\beta_j|}{2(x_k-x_j)}
\end{align}

We start with the simplest slopes to identify on the Newton polygon, the ones with an endpoint at $(0,0)$. The origin must be the bottom left corner of the Newton polygon because $\beta_j > 0$, the slope from it to a point in (eq \ref{eq polygon points}) must be of the form $\gamma_{jk}$, and the convex hull contains the minimum $\gamma_{jk}$.

\begin{definition} \label{defn dominant pair}
    The \textit{dominant pair} of delta potentials, indexed $J,K$, is the one which minimizes the value $\gamma_{jk}$ among all pairs $j,k$. Assumption \ref{assumption} guarantees $J,K$ are uniquely determined by this condition:
    \begin{align} \label{eq gamma minimization}
        \forall j,k \quad \gamma_{JK} \leq \gamma_{jk} \quad \text{with equality iff}\quad 
        J=j,K=k
    \end{align}
\end{definition}

Since slopes on our Newton Polygon are increasing, $\gamma_{JK}$ is the minimum value of $\Gamma$. It corresponds to the string closest to the real axis and hence the longest living resonances. The interpretation of the remaining strings of resonances is given by the following theorem.

\begin{theorem} \label{thm Gamma}
    There exists an integer partitioning of $[1,N]$ i.e. $1=j_1 < ... < j_n = N$ such that the slopes of our Newton polygon are 
    \begin{align} \label{eq Gamma specified}
        \Gamma = \big\{ \gamma_{JK} \big\} \cup \Big\{\tilde{\gamma}_{j_i j_{i+1}} : i=1,...,n-1,\; i\neq I \Big\}
    \end{align}
    These indices include the dominant pair at some $I$ with $j_I=J$, $j_{I+1}=K$.
\end{theorem}
\begin{corollary} \label{cor number strings}
    There are at most $N-1$ strings of resonances. Furthermore, the $K-J-1$ delta functions between the dominant pair cannot contribute strings.
    \begin{align} \label{cor number strings eq}
        \text{\# strings of resonances} \leq N - K + J
    \end{align}
    This inequality is sharp by Example \ref{ex max strings}. Equality holds for $N=2$, $N=3$, and (Example \ref{ex N4}) most cases of $N=4$.
\end{corollary}

Essentially, there is a string of resonances following $-\gamma h\log(1/h)$ if and only if $\gamma$ is a slope on the Newton polygon. Theorem \ref{thm structure} gives the forward implication, while Theorem \ref{thm existence} gives the converse. The distinction between resonances on the dominant string $z_m$ and on the remaining strings $\tilde{z}_m$ becomes clearer in the equations of Corollary \ref{cor log Re z}.

\begin{theorem} \label{thm existence}
    There exists a string of resonances corresponding to every $\gamma \in \Gamma$. More precisely, we have that for any $M$ large, there exists $h_0$ small such that for all $h \in (0,h_0)$, any integer $m$ satisfying $\pi h m / (x_k-x_j) \in [\f{2}{M},\f{M}{2}]$ gives a resonance in one of the following forms.
    \begin{align}
        z_m &= \f{\pi h m}{x_k-x_j} - i \gamma_{jk} h \log(1/h) + \f{ih}{2(x_k-x_j)} \log \left(\f{- C_j C_k (x_k-x_j)^2}{4\pi^2 h^2 m^2}\right) + o(h) \label{thm existence eq dominant} \\
        \tilde{z}_m &= \f{\pi h m}{x_k-x_j} - i \tilde{\gamma}_{jk} h\log(1/h) + \f{ih}{2(x_k-x_j)} \log \left(\f{-C_j}{C_k}\right) + o(h) \label{thm existence eq subdominant}
    \end{align}
    The dominant string is given by (eq \ref{thm existence eq dominant}) and the remaining strings by (eq \ref{thm existence eq subdominant}).
\end{theorem}
\begin{corollary} \label{cor log Re z}
    Resonances on the dominant string are logarithmic with respect to $\Im z$ against $\Re z$, while the remaining strings are flat up to $o(h)$.
    \begin{align}
        \Im z_m &= - \f{h}{x_K-x_J} \log\left( \f{2\Re z_m}{\sqrt{|C_J C_K|}}\right) - \gamma_{JK} h \log(1/h) + o(h) \\
        \Im \tilde{z}_m &= -\f{h}{x_k-x_j} \log \left(\sqrt{\left| \f{C_j}{C_k} \right|}\right) -\tilde{\gamma}_{jk} h \log(1/h) + o(h)
    \end{align}
\end{corollary}

We can now revisit the MATLAB plot in Figure \ref{fig intro setup} as referring to the theory given by Corollary \ref{cor log Re z}. The three strings of resonances here correspond to the three slopes in the Newton polygon of Figure \ref{fig intro polygon}. We provide many more generic and concrete examples below.

\subsection{Examples and discussion}

Our first example sets all $\beta_j$ equal to investigate the role of position $x_j$. We find there is always one string of resonances, resolving a conjecture of \cite{Datchev_Marzuola_Wunsch}. Since the largest distance is preferred, there is resemblance with the longest geodesic behavior of \cite{Hillairet_Wunsch}. Also note that the parameters $C_j$ are insignificant here and are only considered in Example \ref{ex C h}, implying the choice of $h$-dependent barriers is crucial to the result.
\begin{example} \label{ex equal strength}
    Consider equal strength parameters $\beta_1 = ... = \beta_N$. Per Theorem \ref{thm Gamma}, the dominant pair is $J=1$, $K=N$, since $\beta_j+\beta_k$ is constant and $\f{1}{x_k-x_j}$ favors the largest distance. This alone implies there is exactly one string of resonances, since $N-K+J=1$ in Corollary \ref{cor number strings}. See Example \ref{ex N4} part $a$ for a more concrete example.
\begin{figure}[H]
\centering
\resizebox{\textwidth}{!}{
    \begin{tikzpicture} % deltas
    [scale=0.80, ultra thick]
    \draw[->] (0.0,0) -- (0.0, 2);
    \draw[->] (0.5,0) -- (0.5, 2);
    \draw[->] (2,0) -- (2, 2);
    \draw[->] (3,0) -- (3, 2);
    \draw[->] (6,0) -- (6, 2);
    \node[anchor=south] (b1) at (0.0, 2) {$\beta_J$};
    \node[anchor=south] (b2) at (0.5, 2) {$\beta_2$};
    \node[anchor=south] (b3) at (2, 2) {$\beta_3$};
    \node[anchor=south] (b5) at (3, 2) {$\beta_4$};
    \node[anchor=south] (b7) at (6.0, 2) {$\beta_K$};
    \node[anchor=north] (l1) at (0.0, 0) {$x_J$};
    \node[anchor=north] (l1) at (0.5, 0) {$x_2$};
    \node[anchor=north] (l1) at (2,0) {$x_3$};
    \node[anchor=north] (l1) at (3,0) {$x_4$};
    \node[anchor=north] (l1) at (6,0) {$x_K$};
    \draw[decorate, decoration = {brace, amplitude=10pt}] (6,-.5) -- (0,-.5);
    \node (bk) at (4.5,1) {\scalebox{2}{$\cdots{}$}};
    \draw[<->, thin] (-0.5,0) -- (6.5,0);
    \node (placement) at (0,-1) { };
    \end{tikzpicture} 
    \;
    \begin{tikzpicture} % Newton polygon
    [scale=.60, ultra thick]
    \draw[fill=gray!15] (0,6) -- (0,0) -- (6,4) -- (6,6);
    \draw[->, thin] (0,0) -- (0,6);
    \draw[->, thin] (0,0) -- (6,0);
    \node[circle, fill, inner sep=1.5pt] (point) at (0,0) {};
    \node[circle, fill, inner sep=1.5pt] (point) at (0.5,4.0) {};
    \node[circle, fill, inner sep=1.5pt] (point) at (1.0,4.0) {};
    \node[circle, fill, inner sep=1.5pt] (point) at (1.5,4.0) {};
    \node[circle, fill, inner sep=1.5pt] (point) at (2.0,4.0) {};
    \node[circle, fill, inner sep=1.5pt] (point) at (2.5,4.0) {};
    \node[circle, fill, inner sep=1.5pt] (point) at (3.0,4.0) {};
    \node (bk) at (4.5,4) {\scalebox{2}{$\cdots{}$}};
    \node[circle, fill, inner sep=1.5pt] (point) at (6,4.0) {};
    \node[anchor=north] (axes) at (3,0) {$\lambda = 2(x_k - x_j)$};
    \node[anchor=south, rotate=90] (axes) at (0,3) {$\nu = \beta_j + \beta_k$};
    \end{tikzpicture} 
    \;
    \begin{tikzpicture} % resonances
        [scale=0.70, ultra thick]
        \draw[->] (0.5,-1.5) -- (6.5,-1.5);
        \node[anchor=south west] (g1) at (0.5, -1.5) {$\gamma_{JK} = (\beta_1+\beta_N)/2(x_N-x_1)$};
        \draw[->, thin] (0,0) -- (7,0);
        \draw[->, thin] (0,0) -- (0,-5);
        \node[anchor=south] (axes) at (3.5,0) {$\Re z$};
        \node[anchor=south, rotate=90] (axes) at (0,-2.5) {$\Im z$};
        \node[anchor=north] (axes) at (0,-5.5) { }; % placeholder shift
    \end{tikzpicture}
    } 
\end{figure}
\end{example}

\begin{example} \label{ex max strings}
    To achieve the maximum $N-1$ strings in Corollary \ref{cor number strings}, we consider highly unequal strength deltas. If $\beta_j << \beta_{j+1}$ and $x_{j+1}-x_j = 1$ for all $j$, e.g. $\beta_j = 10^{-j}$, then points in (eq \ref{eq polygon points}) are $(\beta_j+\beta_k)/2(x_k-x_j) \approx \beta_k / 2(k-j)$. We show $K-J=1$ in the figure, and it follows that (eq \ref{cor number strings eq}) is sharp for any $J,K$ by placing much weaker strings between the dominant pair $J=1$, $K=2$.
\begin{figure}[H]
    \centering
    \resizebox{\textwidth}{!}{
    \begin{tikzpicture} % deltas
    [scale=0.7, ultra thick]
    \draw[->] (0,0) -- (0, 4);
    \draw[->] (1,0) -- (1, 4/3);
    \draw[->] (2,0) -- (2, 4/9);
    \draw[->] (3,0) -- (3, 0.25);
    \draw[->] (6,0) -- (6, 0.25);
    \node[anchor=south] (b1) at (0, 4) {$\beta_J$};
    \node[anchor=south] (b2) at (1, 4/3) {$\beta_K$};
    \node[anchor=south] (b3) at (2, 4/9) {$\beta_3$};
    \node[anchor=south] (b5) at (3, 0.25) {$\beta_4$};
    \node[anchor=south] (b7) at (6, 0.25) {$\beta_N$};
    \node[anchor=north] (l1) at (0, 0) {$x_J$};
    \node[anchor=north] (l1) at (1, 0) {$x_K$};
    \node[anchor=north] (l1) at (2,0) {$x_3$};
    \node[anchor=north] (l1) at (3,0) {$x_4$};
    \node[anchor=north] (l1) at (6,0) {$x_N$};
    \draw[decorate, decoration = {brace, amplitude=10pt}] (1,-.5) -- (0,-.5);
    \draw[decorate, decoration = {brace, amplitude=3pt}] (2,-.5) -- (1,-.5);
    \draw[decorate, decoration = {brace, amplitude=3pt}] (3,-.5) -- (2,-.5);
    \node (bk) at (4.5,0.5) {\scalebox{2}{$\cdots{}$}};
    \draw[<->, thin] (-0.5,0) -- (6.5,0);
    \node (placement) at (0,-1) { };
    \end{tikzpicture} \;
    \begin{tikzpicture} % Newton polygon
    [scale=.5, ultra thick]
    \draw[fill=gray!15] (0,7) -- (0,0) -- (1,2/27) -- (2,2/9) -- (3,2/3) -- (6,6) -- (6,7);
    \draw[->, thin] (0,0) -- (0,7);
    \draw[->, thin] (0,0) -- (6,0);
    \node[circle, fill, inner sep=1.5pt] (point) at (0,0) {};
    \node[circle, fill, inner sep=1.5pt] (point) at (1,2/27) {};
    \node[circle, fill, inner sep=1.5pt] (point) at (2,2/9) {};
    \node[circle, fill, inner sep=1.5pt] (point) at (3,2/3) {};
    \node[circle, fill, inner sep=1.5pt] (point) at (6,6.0) {};
    \node[circle, fill, inner sep=1.5pt] (point) at (1,2/9) {};
    \node[circle, fill, inner sep=1.5pt] (point) at (1,2/3) {};
    \node[circle, fill, inner sep=1.5pt] (point) at (2,2/3) {};
    \node (bk) at (4.5,5) {\scalebox{2}{$\cdots{}$}};
    \node[anchor=north] (axes) at (3,0) {$\lambda = 2(x_k - x_j)$};
    \node[anchor=south, rotate=90] (axes) at (0,3.5) {$\nu = \beta_j + \beta_k$};
    \end{tikzpicture} \;
    \begin{tikzpicture} % resonances
        [scale=0.6, ultra thick]
        \draw[->] (0.5,-2/3) -- (6.5,-2/3);
        \draw[->] (0.5,-2) -- (6.5,-2);
        \draw[->] (0.5,-6) -- (6.5,-6);
        \node[anchor=north west] (g1) at (0.5, -2/3) {$\gamma_{JK} = (\beta_J+\beta_K)/2(x_K-x_J)$};
        \node[anchor=north west] (g1) at (0.5, -2) {$\tilde{\gamma}_{K3} = (\beta_3-\beta_K)/2(x_3-x_K)$};
        \node[anchor=south west] (g1) at (0.5, -6) {$\tilde{\gamma}_{34} = (\beta_4-\beta_3)/2(x_4-x_3)$};
        \node (bk) at (3.5,-6.5) {\scalebox{2}{$\cdots{}$}};
        \draw[->, thin] (0,0) -- (7,0);
        \draw[->, thin] (0,0) -- (0,-6.5);
        \node[anchor=south] (axes) at (3.5,0) {$\Re z$};
        \node[anchor=south, rotate=90] (axes) at (0,-3.25) {$\Im z$};
        % \node[anchor=north] (axes) at (0,-5.5) {}; % placeholder shift
    \end{tikzpicture}
    } % resize box
\end{figure}
\end{example}

The case of $N=2$ always produces one string of resonances, so we turn to $N=3$ to demonstrate the emergence of multiple strings of resonances with the Newton polygon.
\begin{example} \label{ex N3}
    Consider three delta functions at $x_1=0$, $x_2=4$, $x_3=6$ with asymptotic strength $\beta_1=0.5$, $\beta_2=0.5$, $\beta_3=2.0$. We simplify the other constants by setting $C_1 = C_2 = C_3 = 1$, and we only need a relatively large $h=0.1$ to see agreement with theory. We discuss each of the following three figures in turn.
\begin{enumerate}
    \item We first draw these parameters. The explicit potential is $V(x) = \sum_{j=1}^N C_j h^{\beta_j+1} \delta(x-x_j) \approx 0.03 \delta(x) + 0.03 \delta(x-4) + 0.001 \delta(x-6)$. Among $\gamma_{12} = 0.125$, $\gamma_{13} \approx 0.208$, and $\gamma_{23} = 0.625$, we see that $\gamma_{12}$ is minimal so $J=1$, $K=2$ is dominant.
    \item The Newton polygon is constructed of points $\big(2(x_k-x_j)\, \beta_j+\beta_k\big)$ for all pairs $j<k$. Namely, these $(j,k)$ are $(1,2)$, $(2,3)$, and $(1,3)$, labeled as black dots on the polygon. We also see that the slopes are $\gamma_{12}$ and $\tilde{\gamma}_{23}$.
    \item As predicted, there are two distinct strings of resonances in the numerical computation by MATLAB. We elaborate on why the bottom string here has fewer resonances than the dominant string in Example \ref{ex exact resonances}. The theory lines colored blue correspond to the two blue slopes of the Newton polygon and to the two blue brackets of the potential function sketch.
\end{enumerate}
\begin{figure}[H]
\centering
\resizebox{\textwidth}{!}{
\begin{tikzpicture} % deltas
[scale=0.7, ultra thick]
\draw[->] (0,0) -- (0,4);
\draw[->] (4,0) -- (4,4);
\draw[->] (6,0) -- (6,1);
\node[anchor=south] (b1) at (0,4) {$\beta_1 = 0.5$};
\node[anchor=south] (b2) at (4,4) {$\beta_2 = 0.5$};
\node[anchor=south] (b3) at (6,1) {$\beta_3 = 2.0$};
\node[anchor=south] (l1) at (2,0) {$x_2-x_1=4$};
\node[anchor=south] (l3) at (5,0) {$2$};
\draw[decorate, decoration = {brace, amplitude=10pt}, blue] (4,-0.2) -- (0,-0.2);
\draw[decorate, decoration = {brace, amplitude=3pt}, blue] (6,-0.2) -- (4,-0.2);
\draw[<->, thin] (-0.5,0) -- (6.5,0);
\node (g) at (1.5,-1.2) {$\gamma_{12} = 0.25$};
\node (g) at (5,-0.8) {$\tilde{\gamma}_{23} = .375$};
\end{tikzpicture} 
\begin{tikzpicture} % Newton polygon
    [scale=.6, ultra thick]
    \draw[fill=gray!15] (0,6) -- (0,0) -- (4,1) -- (6,2.5) -- (6,6);
    \draw[blue] (0,0) -- (4,1) -- (6,2.5);
    \draw[->, thin] (0,0) -- (0,6);
    \draw[->, thin] (0,0) -- (6,0);
    \node[anchor=north west] (g) at (2, 0.6) {$\gamma_{12}$};
    \node[anchor=north west] (g) at (5, 1.75) {$\tilde{\gamma}_{23}$};
    \node[circle, fill, inner sep=1.5pt] (point) at (0,0) {};
    \node[circle, fill, inner sep=1.5pt] (point) at (4,1.0) {};
    \node[circle, fill, inner sep=1.5pt] (point) at (6,2.5) {};
    \node[circle, fill, inner sep=1.5pt] (point) at (2,2.5) {};
    \node[anchor=south east] (v) at (4,1) {$(8,1)$};
    \node[anchor=south east] (v) at (6,2.5) {$(12,2.5)$};
    \node[anchor=south east] (v) at (2,2.5) {$(4,2.5)$};
    \node[anchor=north] (axes) at (3,0) {$\lambda = 2(x_k - x_j)$};
    \node[anchor=south, rotate=90] (axes) at (0,3) {$\nu = \beta_j + \beta_k$};
\end{tikzpicture} 
\includegraphics[width=0.3\textwidth]{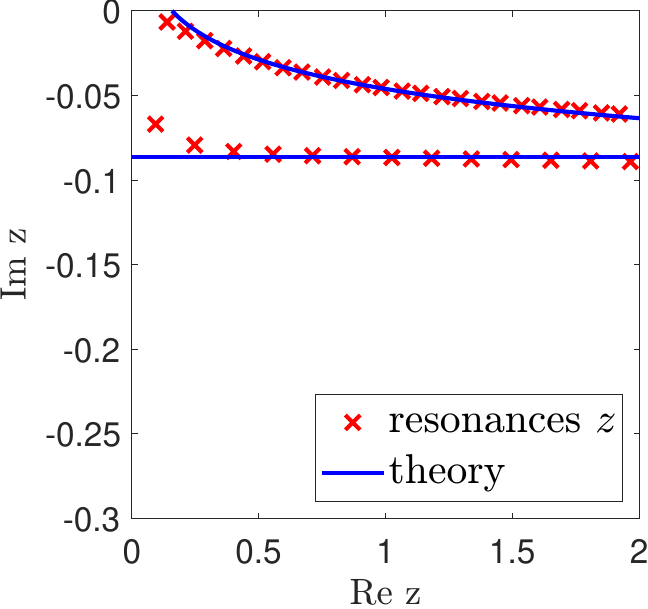} % resonances
} % resize box
\end{figure}
\end{example}

The cases of $N=2$ and $N=3$ were well understood in \cite{Datchev_Marzuola_Wunsch} and \cite{Tanimu} among others, but a wider range of conditions emerges at $N=4$. In fact, it inspired several of our results. It is the first case where the choice of dominant pair does not uniquely give the number of strings of resonances, specifically in case $d$ and $e$ below, yet the dominant pair characterizes cases $a-c$. Hence, it shows how strongly the dominant pair affects the resonances.

\begin{example} \label{ex N4}
    Consider $N=4$. There are five cases of how strings correspond to intervals, up to symmetry, and the dominant pair distinguishes three of them. Brackets correspond to slopes on the Newton polygon, and the largest bracket signifies the dominant pair. We take $h=0.1$ and $C_1 = C_2 = C_3 = C_4 = 1$. Note again that larger $\beta_j$ implies weaker delta functions, signified in how we draw the arrows.
\begin{enumerate}[label=\alph*]
    \item % a
Dominant $J=1$, $K=4$: This is a particular case of Example \ref{ex equal strength}. The middle two delta functions are overshadowed.
\begin{figure}[H]
\centering
\resizebox{\textwidth}{!}{
\begin{tikzpicture}
[scale=0.7, ultra thick] % deltas
    \draw[->] (0,0) -- (0,1);
    \draw[->] (2,0) -- (2,1);
    \draw[->] (5,0) -- (5,1);
    \draw[->] (6,0) -- (6,1);
    \node[anchor=south] (b1) at (0,1) {$\beta_1 = 2.0$};
    \node[anchor=south] (b2) at (2,1) {$\beta_2 = 2.0$};
    \node[anchor=south east] (b3) at (5,1) {$\beta_3 = 2.0$};
    \node[anchor=south] (b4) at (6,1) {$\beta_4 = 2.0$};
    \node[anchor=south] (l1) at (1.0,0) {$2$};
    \node[anchor=south] (l2) at (3.5,0) {$x_3-x_2 = 3$};
    \node[anchor=south] (l3) at (5.5,0) {$1$};
    \draw[decorate, decoration = {brace, amplitude=10pt}] (6,-0.2) -- (0,-0.2);
    \draw[<->, thin] (-0.5,0) -- (6.5,0);
    \node (g) at (3,-1.2) {$\gamma_{14}$};
\end{tikzpicture}
\begin{tikzpicture} % Newton polygon
    [scale=.6, ultra thick]
    \draw[fill=gray!15] (0,6) -- (0,0) -- (6,4) -- (6,6);
    \draw[->, thin] (0,0) -- (0,6);
    \draw[->, thin] (0,0) -- (6,0);
    \node[anchor=north west] (g) at (3,2) {$\gamma_{14}$};
    \node[circle, fill, inner sep=1.5pt] (point) at (0,0) {};
    \node[circle, fill, inner sep=1.5pt] (point) at (1,4.0) {};
    \node[circle, fill, inner sep=1.5pt] (point) at (2,4.0) {};
    \node[circle, fill, inner sep=1.5pt] (point) at (3,4.0) {};
    \node[circle, fill, inner sep=1.5pt] (point) at (4,4.0) {};
    \node[circle, fill, inner sep=1.5pt] (point) at (5,4.0) {};
    \node[circle, fill, inner sep=1.5pt] (point) at (6,4.0) {};
    \node[anchor=north] (axes) at (3,0) {$\lambda = 2(x_k - x_j)$};
    \node[anchor=south, rotate=90] (axes) at (0,3) {$\nu = \beta_j + \beta_k$};
\end{tikzpicture}
\includegraphics[width=0.3\textwidth]{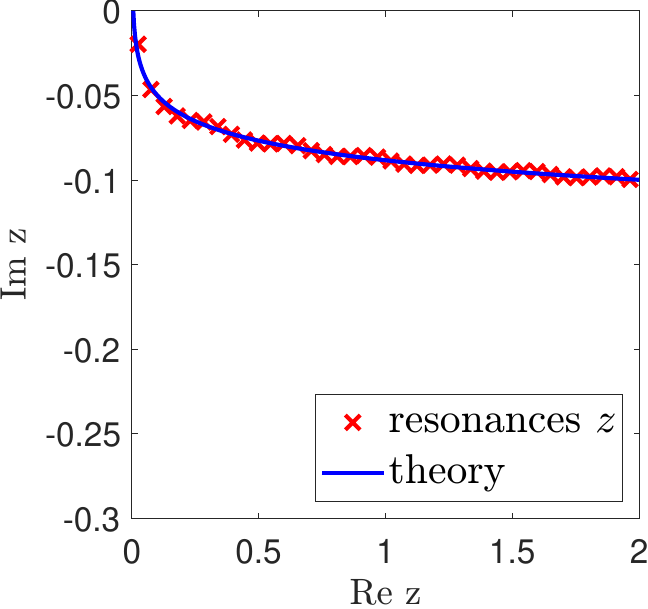} % resonances
} % resize box
\end{figure}
    \item % b
Dominant $J=1$, $K=3$. The strongest delta function by $\beta_3$ pairs with the farthest away of the remaining deltas at $x_1$ to form the dominant pair.
\begin{figure}[H]
\centering
\resizebox{\textwidth}{!}{
\begin{tikzpicture} % deltas
[scale=0.7, ultra thick]
    \draw[->] (0,0) -- (0,1);
    \draw[->] (2,0) -- (2,1);
    \draw[->] (5,0) -- (5,4);
    \draw[->] (6,0) -- (6,1);
    \node[anchor=south] (b1) at (0,1) {$\beta_1 = 2.0$};
    \node[anchor=south] (b2) at (3,1) {$\beta_2 = 2.0$};
    \node[anchor=south] (b3) at (5,4) {$\beta_3 = 0.5$};
    \node[anchor=south] (b4) at (6,1) {$\beta_4 = 2.0$};
    \node[anchor=south] (l1) at (1.0,0) {$2$};
    \node[anchor=south] (l2) at (3.5,0) {$x_3-x_2 = 3$};
    \node[anchor=south] (l3) at (5.5,0) {$1$};
    \draw[decorate, decoration = {brace, amplitude=10pt}] (5,-0.2) -- (0,-0.2);
    \draw[decorate, decoration = {brace, amplitude=3pt}] (6,-0.2) -- (5,-0.2);
    \draw[<->, thin] (-0.5,0) -- (6.5,0);
    \node (g) at (2.5,-1.2) {$\gamma_{13}$};
    \node (g) at (5.5,-0.8) {$\tilde{\gamma}_{34}$};
\end{tikzpicture}
\begin{tikzpicture} % Newton polygon
    [scale=.6, ultra thick]
    \draw[fill=gray!15] (0,6) -- (0,0) -- (5,2.5) -- (6,4) -- (6,6);
    \draw[->, thin] (0,0) -- (0,6);
    \draw[->, thin] (0,0) -- (6,0);
    \node[anchor=north west] (g) at (2.5,1.25) {$\gamma_{13}$};
    \node[anchor=north west] (g) at (5.5,3.25) {$\tilde{\gamma}_{34}$};
    \node[circle, fill, inner sep=1.5pt] (point) at (0,0) {};
    \node[circle, fill, inner sep=1.5pt] (point) at (1,2.5) {};
    \node[circle, fill, inner sep=1.5pt] (point) at (2,4.0) {};
    \node[circle, fill, inner sep=1.5pt] (point) at (3,2.5) {};
    \node[circle, fill, inner sep=1.5pt] (point) at (4,4.0) {};
    \node[circle, fill, inner sep=1.5pt] (point) at (5,2.5) {};
    \node[circle, fill, inner sep=1.5pt] (point) at (6,4.0) {};
    \node[anchor=north] (axes) at (3,0) {$\lambda = 2(x_k - x_j)$};
    \node[anchor=south, rotate=90] (axes) at (0,3) {$\nu = \beta_j + \beta_k$};
\end{tikzpicture}
\includegraphics[width=0.3\textwidth]{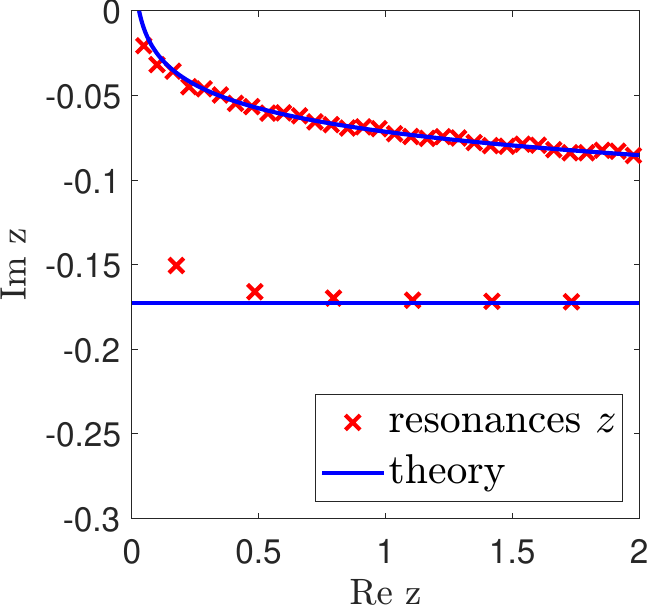} % resonances
}
\end{figure}
    \item % c
Dominant $J=2$, $K=3$: The middle position of the dominant delta functions forces three strings to occur.
\begin{figure}[H]
\centering
\resizebox{\textwidth}{!}{
\begin{tikzpicture} % deltas
[scale=0.7, ultra thick]
    \draw[->] (0,0) -- (0,1);
    \draw[->] (2,0) -- (2,4);
    \draw[->] (5,0) -- (5,4);
    \draw[->] (6,0) -- (6,1);
    \node[anchor=south] (b1) at (0,1) {$\beta_1 = 2.0$};
    \node[anchor=south] (b2) at (2,4) {$\beta_2 = 0.5$};
    \node[anchor=south] (b3) at (5,4) {$\beta_3 = 0.5$};
    \node[anchor=south] (b4) at (6,1) {$\beta_4 = 2.0$};
    \node[anchor=south] (l1) at (1.0,0) {$2$};
    \node[anchor=south] (l2) at (3.5,0) {$x_3-x_2 = 3$};
    \node[anchor=south] (l3) at (5.5,0) {$1$};
    \draw[decorate, decoration = {brace, amplitude=10pt}] (5,-0.2) -- (2,-0.2);
    \draw[decorate, decoration = {brace, amplitude=3pt}] (2,-0.2) -- (0,-0.2);
    \draw[decorate, decoration = {brace, amplitude=3pt}] (6,-0.2) -- (5,-0.2);
    \draw[<->, thin] (-0.5,0) -- (6.5,0);
    \node (g) at (1,-0.8) {$\tilde{\gamma}_{12}$};
    \node (g) at (3.5,-1.2) {$\gamma_{23}$};
    \node (g) at (5.5,-0.8) {$\tilde{\gamma}_{34}$};
\end{tikzpicture}
\begin{tikzpicture} % Newton polygon
    [scale=.7, ultra thick]
    \draw[fill=gray!15] (0,6) -- (0,0) -- (3,1) -- (5,2.5) -- (6,4) -- (6,6);
    \draw[->, thin] (0,0) -- (0,6);
    \draw[->, thin] (0,0) -- (6,0);
    \node[anchor=north west] (g) at (1.5,0.5) {$\gamma_{23}$};
    \node[anchor=north west] (g) at (4,1.75) {$\tilde{\gamma}_{12}$};
    \node[anchor=north west] (g) at (5.5,3.25) {$\tilde{\gamma}_{34}$};
    \node[circle, fill, inner sep=1.5pt] (point) at (0,0) {};
    \node[circle, fill, inner sep=1.5pt] (point) at (1,2.5) {};
    \node[circle, fill, inner sep=1.5pt] (point) at (2,2.5) {};
    \node[circle, fill, inner sep=1.5pt] (point) at (3,1.0) {};
    \node[circle, fill, inner sep=1.5pt] (point) at (4,2.5) {};
    \node[circle, fill, inner sep=1.5pt] (point) at (5,2.5) {};
    \node[circle, fill, inner sep=1.5pt] (point) at (6,4.0) {};
    \node[anchor=north] (axes) at (3,0) {$\lambda = 2(x_k - x_j)$};
    \node[anchor=south, rotate=90] (axes) at (0,3) {$\nu = \beta_j + \beta_k$};
\end{tikzpicture}
\includegraphics[width=0.30\textwidth]{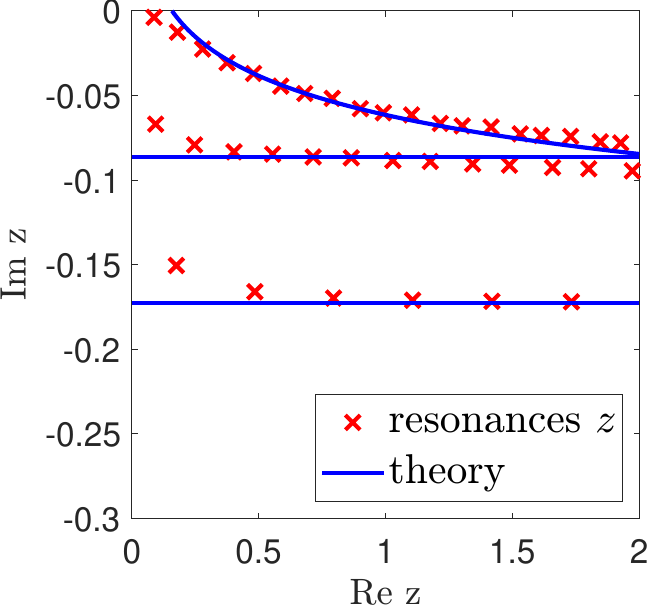} % resonances
} % resize box
\end{figure}
    \item % d
Dominant $J=1$, $K=2$ with 2 strings: This is the only case for $N=4$ with strict inequality in (eq \ref{cor number strings eq}), since there are 2 strings but $2 < 4 - 2 + 1 = 3$. It is also the only case where the dominant pair does not uniquely determine the number of strings, by comparison with the subsequent part e.
\begin{figure}[H]
\centering
\resizebox{\textwidth}{!}{
\begin{tikzpicture} % deltas
[scale=0.7, ultra thick]
\draw[->] (0,0) -- (0,4);
\draw[->] (3,0) -- (3,4);
\draw[->] (4,0) -- (4,1);
\draw[->] (6,0) -- (6,.66);
\node[anchor=south] (b1) at (0,4) {$\beta_1 = 0.5$};
\node[anchor=south] (b2) at (3,4) {$\beta_2 = 0.5$};
\node[anchor=south] (b3) at (4,1) {$\beta_3 = 2.0$};
\node[anchor=south] (b4) at (6,.66) {$\beta_4 = 3.0$};
\node[anchor=south] (l1) at (1.5,0) {$x_2-x_1=3$};
\node[anchor=south] (l2) at (3.5,0) {$1$};
\node[anchor=south] (l3) at (5,0) {$2$};
\draw[decorate, decoration = {brace, amplitude=10pt}] (3,-0.2) -- (0,-0.2);
\draw[decorate, decoration = {brace, amplitude=3pt}] (6,-0.2) -- (3,-0.2);
\draw[<->, thin] (-0.5,0) -- (6.5,0);
\node (g) at (1.5,-1.2) {$\gamma_{12}$};
\node (g) at (4.5,-0.8) {$\tilde{\gamma}_{24}$};
\end{tikzpicture}
\begin{tikzpicture} % Newton polygon
    [scale=.7, ultra thick]
    \draw[fill=gray!15] (0,6) -- (0,0) -- (3,1) -- (6,3.5) -- (6,6);
    \draw[->, thin] (0,0) -- (0,6);
    \draw[->, thin] (0,0) -- (6,0);
    \node[anchor=north west] (g) at (1.5,0.5) {$\gamma_{12}$};
    \node[anchor=north west] (g) at (4.5,2.25) {$\tilde{\gamma}_{24}$};
    \node[circle, fill, inner sep=1.5pt] (point) at (0,0) {};
    \node[circle, fill, inner sep=1.5pt] (point) at (3,1.0) {};
    \node[circle, fill, inner sep=1.5pt] (point) at (1,2.5) {};
    \node[circle, fill, inner sep=1.5pt] (point) at (4,2.5) {};
    \node[circle, fill, inner sep=1.5pt] (point) at (3,3.5) {};
    \node[circle, fill, inner sep=1.5pt] (point) at (6,3.5) {};
    \node[circle, fill, inner sep=1.5pt] (point) at (2,5.0) {};
    \node[anchor=north] (axes) at (3,0) {$\lambda = 2(x_k - x_j)$};
    \node[anchor=south, rotate=90] (axes) at (0,3) {$\nu = \beta_j + \beta_k$};
\end{tikzpicture}
\includegraphics[width=0.3\textwidth]{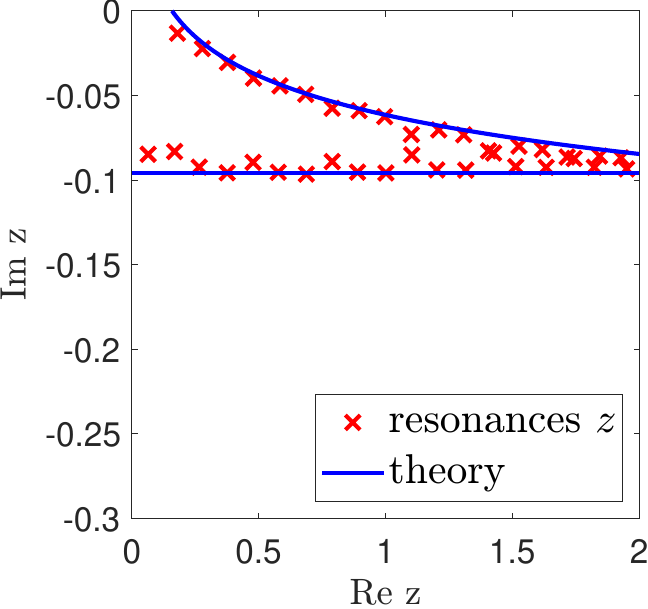} % resonances
} % resize box
\end{figure}
    \item % e
Dominant $J=1$, $K=2$ with 3 strings: Only a slight change in the point $(4,2.5)$ in the previous case to $(5,2.5)$ is needed to change the number of strings. Nevertheless, the slopes of the polygon are hard to distinguish, and the strings are correspondingly close.
\begin{figure}[H]
\centering
\resizebox{\textwidth}{!}{
\begin{tikzpicture} % deltas
[scale=0.7, ultra thick]
\draw[->] (0,0) -- (0,4);
\draw[->] (3,0) -- (3,4);
\draw[->] (5,0) -- (5,1);
\draw[->] (6,0) -- (6,.66);
\node[anchor=south] (b1) at (0,4) {$\beta_1 = 0.5$};
\node[anchor=south] (b2) at (3,4) {$\beta_2 = 0.5$};
\node[anchor=south] (b3) at (5,1.2) {$\beta_3 = 2.0$};
\node[anchor=south] (b4) at (6,.66) {$\beta_4 = 3.0$};
\node[anchor=south] (l1) at (1.5,0) {$x_2-x_1=3$};
\node[anchor=south] (l2) at (4,0) {$2$};
\node[anchor=south] (l3) at (5.5,0) {$1$};
\draw[decorate, decoration = {brace, amplitude=10pt}] (3,-0.2) -- (0,-0.2);
\draw[decorate, decoration = {brace, amplitude=3pt}] (5,-0.2) -- (3,-0.2);
\draw[decorate, decoration = {brace, amplitude=3pt}] (6,-0.2) -- (5,-0.2);
\draw[<->, thin] (-0.5,0) -- (6.5,0);
\node (g) at (1.5,-1.2) {$\gamma_{12}$};
\node (g) at (4.0,-0.8) {$\tilde{\gamma}_{23}$};
\node (g) at (5.5,-0.8) {$\tilde{\gamma}_{34}$};
\end{tikzpicture}
\begin{tikzpicture} % Newton polygon
    [scale=.7, ultra thick]
    \draw[fill=gray!15] (0,6) -- (0,0) -- (3,1) -- (5,2.5) -- (6,3.5) -- (6,6);
    \draw[->, thin] (0,0) -- (0,6);
    \draw[->, thin] (0,0) -- (6,0);
    \node[anchor=north west] (g) at (1.5,0.5) {$\gamma_{12}$};
    \node[anchor=north west] (g) at (4.0,1.75) {$\tilde{\gamma}_{23}$};
    \node[anchor=north west] (g) at (5.5,3.0) {$\tilde{\gamma}_{34}$};
    \node[circle, fill, inner sep=1.5pt] (point) at (0,0) {};
    \node[circle, fill, inner sep=1.5pt] (point) at (3,1.0) {};
    \node[circle, fill, inner sep=1.5pt] (point) at (2,2.5) {};
    \node[circle, fill, inner sep=1.5pt] (point) at (5,2.5) {};
    \node[circle, fill, inner sep=1.5pt] (point) at (3,3.5) {};
    \node[circle, fill, inner sep=1.5pt] (point) at (6,3.5) {};
    \node[circle, fill, inner sep=1.5pt] (point) at (1,5.0) {};
    \node[anchor=north] (axes) at (3,0) {$\lambda = 2(x_k - x_j)$};
    \node[anchor=south, rotate=90] (axes) at (0,3) {$\nu = \beta_j + \beta_k$};
\end{tikzpicture}
\includegraphics[width=0.3\textwidth]{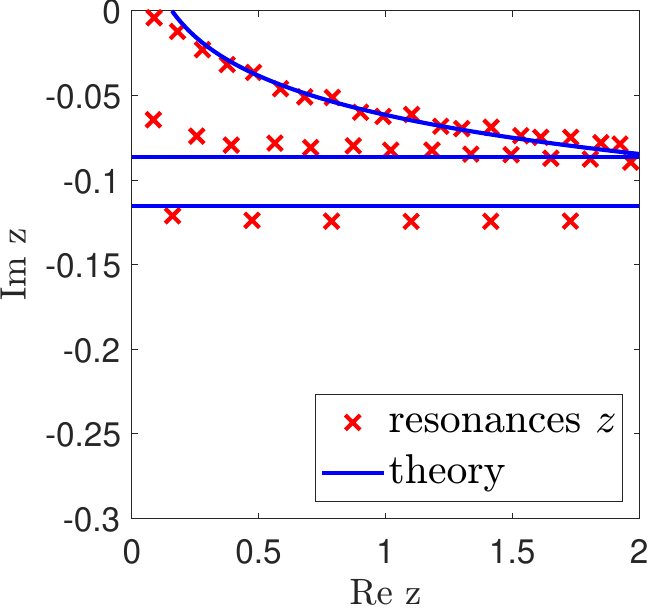} % resonances
} % resize box
\end{figure}
\end{enumerate}
\end{example}

We continue our discussion with open problems. A technical limitation of the current results is that we only consider resonances of the form $\Im z = -\gamma h\log(1/h) + O(h)$, but numerical evidence suggest these are the only ones that exist.
\begin{conjecture}
    Theorem \ref{thm structure} applies to all resonances such that $\Re z \in U \Subset (0,\infty)$ and $\Im z < 0$.
\end{conjecture}
Another observed but omitted behavior is the regime $\Re z \rightarrow \infty$. This would help explain the deviation from theory when predicted strings intersect and seemingly merge, elaborated in Example \ref{ex intersection}. This limit seems to compete with the limit $h\rightarrow 0$.
\begin{conjecture} \label{conj large Re z}
As $\Re z \rightarrow \infty$, all strings coalesce into a single log curve corresponding to the outermost delta functions.
\begin{align}
    \Im z = -\f{h}{x_N - x_1} \log \Re z + O(1)
\end{align}
\end{conjecture}
In addition to the two conjectures above and perhaps more importantly, we have two broader questions. The first concerns a deeper understanding of strings of resonances beyond values of $\gamma$, and the second would strengthen the pattern of geodesics and strengths of singularity discussed in the introduction.

\begin{question}
How do strings of resonances correspond to intervals? We found that the parameters determining the equation of a string correspond to intervals, but it is unclear if this correspondence holds for the wave function. For instance, it would be interesting if the resonant state of resonances on different strings would be concentrated in those respective intervals.
\end{question}
\begin{question}
Do the results generalize to delta functions on graphs? Strong delta functions can obscure weaker ones from contributing strings, but this generalization would determine whether the effect is from collinearity or more organic to the wave function. Two ideas to test the role of distance and strength are in Figure \ref{fig delta graphs}.
\end{question}
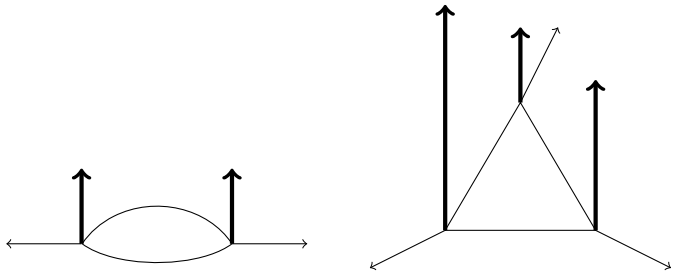
\begin{figure}[H]
    \centering
\begin{tikzpicture}
    [scale=1, ultra thick]
    % axes
    \draw[thin] (2,0) arc[start angle=30, end angle=150, radius=1, x radius=1.15, y radius=1];
    \draw[thin] (2,0) arc[start angle=-30, end angle=-150, radius=1, x radius=1.15, y radius=0.5];
    \draw[thin, ->] (0,0) -- (-1,0);
    \draw[thin, ->] (2,0) -- (3,0);
    % deltas
    \draw[->] (0,0) -- (0,1);
    \draw[->] (2,0) -- (2,1);
\end{tikzpicture} \qquad
\begin{tikzpicture}
    [scale=1, ultra thick]
    % axes
    \draw[thin] (0,0) -- (2,0) -- (1,1.7) -- (0,0);
    \draw[thin, ->] (0,0) -- (-1,-0.5);
    \draw[thin, ->] (2,0) -- (3,-0.5);
    \draw[thin, ->] (1,1.7) -- (1.5,2.7);
    % deltas
    \draw[->] (0,0) -- (0,3);
    \draw[->] (2,0) -- (2,2);
    \draw[->] (1,1.7) -- (1,2.7);
\end{tikzpicture}
    \caption{Two examples of delta functions on graphs of interest to an open problem. The left case considers equal strength deltas connected by unequal length line segments, while the right case considers equal length line segments connecting unequal strength deltas.}
    \label{fig delta graphs}
\end{figure}

We end on a highly speculative connection to the more physics oriented literature.
\begin{question} \label{q bound states}
What about bound states? If bound states exists by $C_j < 0$, there is a tempting connection between the at most $N-1$ strings of resonances and the at most $N$ bound states \cite{Erman_et_al_2017}. In particular, \cite{Erman_et_al_2018} describes $N-1$ threshold anomalies corresponding to creations of new bound states, where the reflection probability goes to zero. We compare that the reflection probability is asymptotically 0 in our setup, as discussed around (eq \ref{RT explanation}). In \cite{Erman_et_al_2018}, they explicitly compute examples of equal strength $\lambda$ Dirac deltas at distance $a$, finding a threshold anomaly at $a = \f{\hbar^2}{2 m\lambda}$ for $N=2$ and at $m\lambda a \in \big\{\hbar^2 (2-\sqrt{2})/4\, \hbar^2 / 2\, \hbar^2 (2+\sqrt{2})/4\big\}$ for $N=4$. These seem to translate into our units by $h^2=\hbar^2/2m$, $\lambda = C h^{\beta+1}$, and $l = a$, yielding the equations 
\begin{align*}
    l\* C h^{\beta-1} &= 1 \;\text{for } N=2 , \\
    l\* C h^{\beta-1} &\in \big\{1 - \sqrt{2}/2\, 1\, 1 + \sqrt{2}/2\big\} \;\text{for } N=4 .
\end{align*}
While the same type of distance and strength parameters appear in both, we do not see a more concrete connection here.
\end{question}

\subsection{Outline}

Section 1 sketched the mathematical setup and provided background, and section 2 precisely stated the main results. The remainder of the paper contains the proofs, besides the code included in section 9.

\begin{enumerate}
    \setcounter{enumi}{2}
    \item The continuity and jump conditions which characterize Dirac deltas give an algebraic equation for the resonances in the form of a large determinant. This setup is a routine calculation.
    \item We explicitly compute this determinant in Lemma \ref{lemma det} by simplifying into a tridiagonal form reminiscent of a Jacobi matrix, then applying induction. The primary difficulty is choosing notation suitable to manage it.
    \item We use the assumptions in Theorem \ref{thm structure} to asymptotically simplify the determinant in Lemma \ref{lemma asymptotic}. The Newton polygon arises naturally from an equation of asymptotically differing terms having at least two largest terms to balance out.
    \item Theorem \ref{thm Gamma} on the Newton polygon depends only on its construction and the geometric setup of the delta functions, not the wave function. Nevertheless, it is ultimately this geometric argument that gives the interpretation of the strings of resonances, such as Corollary \ref{cor number strings}.
    \item Existence of resonances in Theorem \ref{thm existence} is given by Rouch\'e's theorem. A mixed rational and exponential equation is simplified by precisely finding the largest terms from Lemma \ref{lemma polygon} and Theorem \ref{thm Gamma}.
    \item We complete the proof of Theorem \ref{thm existence} by refining the asymptotics.
    \item Lastly, we discuss the numerical fit with theory and provide the code to numerically compute these resonances.
\end{enumerate}

\section{Delta functions: the algebraic equation for the resonances}

This section sketches the routine derivation of the equation determining $z$. See chapter 2 of \cite{Griffiths} for an introduction to delta functions as potentials and chapter 2 of \cite{dyatlov_zworski} for background on resonances. The ansatz below solves the Schr\"odinger equation (eq \ref{resonantState}) in the sense of distributions:
\begin{align} \label{ansatz}
    &u(x) = v_j^+ e^{izx/h} + v_j^- e^{-izx/h} \text{ in }(x_j,x_{j+1}) 
\end{align}
The unknown coefficients $v_j^\pm$ are associated right right-travelling $v^+$ and left-travelling $v^-$ components in the time-dependent equation, depicted in Figure \ref{fig v^pm directions}. For convenience, define the endpoints in (eq \ref{ansatz}) as $x_0 = -\infty$ and $x_{N+1} = \infty$.
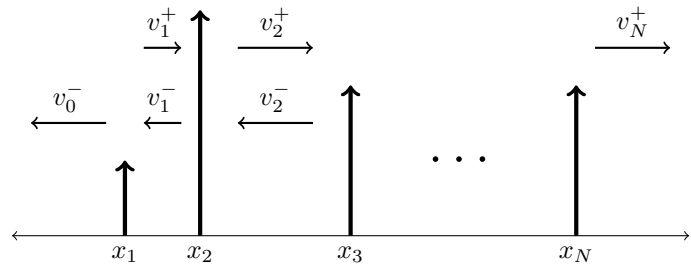
\begin{figure}[H]
\centering
\begin{tikzpicture}
[scale=1, ultra thick]
\draw[->] (0,0) -- (0,1);
\draw[->] (1,0) -- (1,3);
\draw[->] (3,0) -- (3,2);
\draw[->] (6,0) -- (6,2);
\node (bk) at (4.5,1) {\scalebox{2}{$\cdots{}$}};
\draw[<-,thick] (-1.25,1.5) -- (-0.25,1.5); % (-infty, x1)
\node[anchor=south] (v1-) at (-0.75,1.5) {$v_0^-$};
\draw[<-,thick] ( 0.25,1.5) -- (0.75,1.5); % (x1, x2)
\node[anchor=south] (v1-) at (0.50,1.5) {$v_1^-$};
\draw[->,thick] ( 0.25,2.5) -- (0.75,2.5);
\node[anchor=south] (v1+) at (0.50,2.5) {$v_1^+$};
\draw[<-,thick] (1.5,1.5) -- (2.5,1.5); % (x2,x3)
\node[anchor=south] (v1-) at (2,1.5) {$v_2^-$};
\draw[->,thick] (1.5,2.5) -- (2.5,2.5);
\node[anchor=south] (v1+) at (2,2.5) {$v_2^+$};
\draw[->,thick] (6.25,2.5) -- (7.25,2.5); % (xN,infty)
\node[anchor=south] (v1+) at (6.75,2.5) {$v_N^+$};
\node[anchor=north] (x1) at (0,0) {$x_1$};
\node[anchor=north] (x2) at (1,0) {$x_2$};
\node[anchor=north] (x3) at (3,0) {$x_3$};
\node[anchor=north] (x4) at (6,0) {$x_N$};
\draw[<->, thin] (-1.5,0) -- (7.5,0);
\end{tikzpicture}
\caption{General setup of Figure \ref{fig intro setup} annotated with coefficients from the ansatz (eq \ref{ansatz}). Between delta functions as up arrows, coefficients $v_j^+$ correspond to right travelling waves and $v_j^-$ to left travelling waves.}
\label{fig v^pm directions}
\end{figure}
\noindent Resonances are restricted to outgoing solutions, defined as having no incoming component from the unbounded edges, implying $v_0^+ = v_N^- = 0$ but not all $v_j^\pm$ are zero. The remaining coefficients are determined by continuity and jump conditions created by the delta functions:
\begin{align}
    \begin{cases} u(x_j^-) - u(x_j^+) = 0 \\
    h^2\big(u'(x_j^-) - u'(x_j^+)\big) + C_j h^{1+\beta_j} u(x_j) = 0
    \end{cases}
\end{align}
We use the ansatz (eq \ref{ansatz}) to write this more explicitly and group terms by the unknown coefficients $v_j^\pm$.
\begin{align*} 
    \begin{cases}
    v_{j-1}^- e^{-ix_jz/h} + v_{j-1}^+ e^{ix_jz/h} - v_j^- e^{-ix_jz/h} - v_j^+ e^{ix_jz/h} = 0 \\
    - iz v_{j-1}^- e^{-ix_jz/h} + iz v_{j-1}^+ e^{ix_jz/h} + (iz + C_j h^{\beta_j}) v_j^- e^{-ix_jz/h} + (-iz + C_j h^{\beta_j}) v_j^+ e^{ix_jz/h} = 0 
    \end{cases}
\end{align*}
We abbreviate the exponential term with the notation
\begin{align}
    \omega := e^{iz/h}
\end{align}
cautioning that $\omega = w^{-1}$ compared to the notation of \cite{Datchev_Marzuola_Wunsch}. We further simplify it by multiplying through $e^{-ix_j z/h}$:
\begin{align*} 
    \begin{cases}
    v_{j-1}^- \omega^{-2x_j} + v_{j-1}^+ - v_j^- \omega^{-2x_j} - v_j^+ = 0 \\
    - iz v_{j-1}^- \omega^{-2x_j} + iz v_{j-1}^+ + (iz + C_j h^{\beta_j}) v_j^- \omega^{-2x_j} + (-iz + C_j h^{\beta_j}) v_j^+ = 0
    \end{cases}
\end{align*}
We write this system in matrix notation to emphasize its linearity with respect to $v_j^\pm$. Also, recall that coefficients $v_0^+, v_N^-$ are omitted, and the order of the remaining $v_0^-, v_1^-, v_1^+, ...$ is chosen for later convenience. 

\resizebox{0.95\textwidth}{!} {$
    \begin{pmatrix}
    \omega^{-2x_1} & -\omega^{-2x_1} & -1 & 0 & 0 & \hdots & 0 \\
    -iz \omega^{-2x_1} & (iz+C_1h^{\beta_1}) \omega^{-2x_1} & (-iz+C_1h^{\beta_1}) & 0 & 0 & \hdots & 0 \\ \vspace{-7pt}
    0 & \omega^{-2x_2} & 1 & -\omega^{-2x_2} & -1 & & 0 \\ \vspace{-7pt}
    0 & -iz \omega^{-2x_2} & iz & (iz+C_2h^{\beta_2}) \omega^{-2x_2} & (-iz+C_2h^{\beta_2}) & \ddots & 0 \\ 
    \vdots & \vdots & \vdots & \vdots & \vdots & \ddots & -1 \\
    0 & 0 & 0 & 0 & 0 & & (-iz+C_N h^{\beta_N})
    \end{pmatrix}
    \begin{pmatrix}
    v_0^- \\
    v_1^- \\ v_1^+ \\
    v_2^- \\ v_2^+ \\
    \vdots \\
    v_N^+
    \end{pmatrix}
    = 0 
$} \\
We have $2N$ equations, two for each delta function, and $2N$ unknowns $v_j^\pm$, two for each $N+1$ interval between delta functions but excluding $v_0^+ = v_N^- = 0$. To find a nontrivial solution to the last unknown $z$, we require the null space against the $v_j^\pm$ vector is trivial, or equivalently that the determinant of the coefficients of $v_j^\pm$ in matrix form is 0.

\begin{equation}
    \label{det system for z}
    \resizebox{\textwidth}{!}{% 
    $
    \begin{aligned}
        D_N := 
        \det
        \begin{pmatrix}
        \omega^{-2x_1} & -\omega^{-2x_1} & -1 & 0 & 0 & \hdots & 0 \\
        -iz \omega^{-2x_1} & (iz+C_1h^{\beta_1}) \omega^{-2x_1} & (-iz+C_1h^{\beta_1}) & 0 & 0 & \hdots & 0 \\ \vspace{-7pt}
        0 & \omega^{-2x_2} & 1 & -\omega^{-2x_2} & -1 & & 0 \\ \vspace{-7pt}
        0 & -iz \omega^{-2x_2} & iz & (iz+C_2h^{\beta_2}) \omega^{-2x_2} & (-iz+C_2h^{\beta_2}) & \ddots & 0 \\ 
        \vdots & \vdots & \vdots & \vdots & \vdots & \ddots & -1 \\
        0 & 0 & 0 & 0 & 0 & & (-iz+C_N h^{\beta_N})
        \end{pmatrix}
    \end{aligned} 
    $ }
\end{equation}

The remainder of the paper is essentially an analysis of this equation.

\section{Determinant formula: simplifying the equation for resonances}

A mundane but critical step in investigating $z$ is simplifying the determinant formula in (eq \ref{det system for z}). This is primarily an exercise in notation, starting with the following:
\begin{align} \label{RT notation}
    R_j &:= \f{C_j h^\beta_j}{2iz - C_j h^\beta_j} &
    T_j &:= \f{2iz}{2iz - C_j h^\beta_j}
\end{align}
The variables $R_j$ and $T_j$ are suggestively named for reflection and transmission, and their interpretation is provided later in (eq \ref{RT explanation}). We first present the main result of this  section, along with the notation convenient to succinctly prove it.
\begin{lemma} \label{lemma det}
    The determinant given by (eq \ref{det system for z}) is the following, aided by the notation above and the index set below.
    \begin{align}
    \f{1}{c} D_N &= 1 + \sum_{n=1}^{\lfloor N/2\rfloor} (-1)^{n} \sum_{\mathscr{I}_n^N} R_{j_1}^{k_1} \* ... \* R_{j_n}^{k_n} \label{lemma det eq} \\ 
    R_j^k &:= \begin{cases} R_j R_k \omega^{2(x_k-x_j)} & k = j+1 \\ R_j (T_{j+1}+R_{j+1}) \* ... \* (T_{k-1}+R_{k-1}) R_k \omega^{2(x_k-x_j)} & k > j+1 \end{cases} \\
    \mathscr{I}_n^N &:= \Big\{\{j_1,k_1,...,j_n,k_n\} : j_i,k_i \in \mathbb{N} \text{ and } \forall i=1,...,(n-1) \; 1 \leq j_i < k_i < j_{i+1} < k_{i+1} \leq N\Big\} \\
    c &:= - \prod_{j=1}^N \f{\omega^{-2x_j}}{2iz-C_1h^{\beta_1}}
\end{align}
\end{lemma}
\begin{example} \label{ex N2 det}
    We demonstrate (eq \ref{lemma det eq}) for $N=2$ below. With $n=1$ and $\mathscr{I}_1^2 = \big\{ \{1,2\} \}$, there are just two terms.
    \begin{align*}
        \f1c D_2 = 1 - R_1^2 = 1 - R_1 R_2 \omega^{2(x_2-x_1)} &= 0 \\
        \f{C_1 h^{\beta_1}}{2iz-C_1 h^{\beta_1}} \f{C_2 h^{\beta_2}}{2iz-C_2 h^{\beta_2}} e^{2i(x_2-x_1)z/h} &= 1
    \end{align*}
\end{example}
\begin{example}
    We compute $D_4$ to demonstrate the terms from $n=2$. Fortunately, there is only one element in the next level index set $\mathscr{I}_2^4 = \big\{ \{1,2,3,4\} \big\}$.
    \begin{align*}
        \f1c D_4 &= 1 - \sum_{\mathscr{I}_1^4} R_{j_1}^{k_1} + \sum_{\mathscr{I}_2^4} R_{j_1}^{k_1} \* R_{j_2}^{k_2} \\
        &= 1 - \Big(R_1R_2 \omega^{2(x_2-x_1)} + R_2R_3 \omega^{2(x_3-x_2)} + R_3R_4 \omega^{2(x_4-x_3)} + R_1(R_2+T_3)R_3 \omega^{2(x_3-x_1)} + \\ 
        &\quad + R_2(R_3+T_3)R_4 \omega^{2(x_4-x_2)} + R_1(R_2+T_2)(R_3+T_3)R_4 \omega^{2(x_4-x_1)}\Big) \\
        &\quad + R_1R_2R_3R_4 \omega^{2(x_2-x_1 + x_4-x_3)}
    \end{align*}
\end{example}
We now proceed with the proof of Lemma \ref{lemma det}. It is sufficient for our later use that some products of $R_j^k$ are the terms that appear in $D_N$, but we describe the full formula for concreteness. There are two distinct steps: row reduction to tridiagonal form and induction from its recursive relationship.

\begin{proof}
We first simplify (eq \ref{det system for z}) with row reduction to a tridiagonal form. The steps are adding the odd rows $\* iz$ into the even rows, then adding the even rows $\* \f{1}{-2iz+C_1h^{\beta_1}}$ into the odd rows.

\begin{equation*} \resizebox{\textwidth}{!}{$
\begin{aligned}
    & \det \begin{pmatrix}
    \omega^{-2x_1} & -\omega^{-2x_1} & -1 & 0 & 0 & \hdots & 0 \\
    -iz \omega^{-2x_1} & (iz+C_1h^{\beta_1}) \omega^{-2x_1} & (-iz+C_1h^{\beta_1}) & 0 & 0 & \hdots & 0 \\ \vspace{-7pt}
    0 & \omega^{-2x_2} & 1 & -\omega^{-2x_2} & -1 & & 0 \\ \vspace{-7pt}
    0 & -iz \omega^{-2x_2} & iz & (iz+C_2h^{\beta_2}) \omega^{-2x_2} & (-iz+C_2h^{\beta_2}) & \ddots & 0 \\ 
    \vdots & \vdots & \vdots & \vdots & \vdots & \ddots & -1 \\
    0 & 0 & 0 & 0 & 0 & & (-iz+C_N h^{\beta_N})
    \end{pmatrix} \\
    =& 
    \det \begin{pmatrix}
    \omega^{-2x_1} & -\omega^{-2x_1} & -1 & 0 & 0 & \hdots & 0 \\
    0 & C_1h^{\beta_1} \omega^{-2x_1} & (-2iz+C_1h^{\beta_1}) & 0 & 0 & \hdots & 0 \\ \vspace{-7pt}
    0 & \omega^{-2x_2} & 1 & -\omega^{-2x_2} & -1 & & 0 \\ \vspace{-7pt}
    0 & 0 & 2iz & C_2h^{\beta_2} \omega^{-2x_2} & (-2iz+C_2h^{\beta_2}) & \ddots & 0 \\ 
    \vdots & \vdots & \vdots & \vdots & \vdots & \ddots & -1 \\
    0 & 0 & 0 & 0 & 0 & & (-2iz+C_N h^{\beta_N})
    \end{pmatrix} \\
    =&
    \det \begin{pmatrix}
    \omega^{-2x_1} & \f{2iz}{-2iz+C_1 h^{\beta_1}}\omega^{-2x_1} & 0 & 0 & 0 & \hdots & 0 \\
    0 & C_1h^{\beta_1} \omega^{-2x_1} & (-2iz+C_1h^{\beta_1}) & 0 & 0 & \hdots & 0 \\ \vspace{-7pt}
    0 & \omega^{-2x_2} & \f{C_2 h^{\beta_2}}{-2iz + C_2 h^{\beta_2}} & \f{2iz}{-2iz+C_2h^{\beta_2}} \omega^{-2x_2} & 0 & & 0 \\ \vspace{-7pt}
    0 & 0 & 2iz & C_2h^{\beta_2} \omega^{-2x_2} & (-2iz+C_2h^{\beta_2}) & \ddots & 0 \\ 
    \vdots & \vdots & \vdots & \vdots & \vdots & \ddots & 0 \\
    0 & 0 & 0 & 0 & 0 & & (-2iz+C_N h^{\beta_N})
    \end{pmatrix}
\end{aligned}
$} \end{equation*}
Switching to the notation (eq \ref{RT notation}) by outside scaling greatly simplifies the entries. This includes multiplying every row by $-1$, which cancels due to the even number of rows.
\begin{align*}
    D_N =& \prod_{j=1}^N \f{1}{2iz-C_j h^{\beta_j}} \*
    \det \begin{pmatrix}
    \omega^{-2x_1} & -T_1 \omega^{-2x_1} & 0 & 0 & 0 & \hdots & 0 \\
    0 & R_1 \omega^{-2x_1} & -1 & 0 & 0 & \hdots & 0 \\ \vspace{-7pt}
    0 & \omega^{-2x_2} & -R_2 & -T_2 \omega^{-2x_2} & 0 & & 0 \\ \vspace{-7pt}
    0 & 0 & T_2 & R_2 \omega^{-2x_2} & -1 & \ddots & 0 \\ 
    \vdots & \vdots & \vdots & \vdots & \vdots & \ddots & 0 \\
    0 & 0 & 0 & 0 & 0 & & -1
    \end{pmatrix} \\
    =& -\prod_{j=1}^N \f{\omega^{-2x_j}}{2iz - C_j h^{\beta_j}} \*
    \det \begin{pmatrix}
    R_1 \omega^{-2x_1} & -1 & 0 & 0 & \hdots \\ \vspace{-7pt}
    1 & -R_2 \omega^{2x_2} & -T_2 & 0 &  \\
    0 & T_2 & R_2 \omega^{-2x_2} & -1 & \ddots  \\ 
    \vdots & \vdots & \vdots & \vdots & \ddots 
    \end{pmatrix} \\
    =& c \det \begin{pmatrix}
    R_1 \omega^{-2x_1} & -1 & 0 & 0 & \hdots \\ \vspace{-7pt}
    1 & -R_2 \omega^{2x_2} & -T_2 & 0 &  \\
    0 & T_2 & R_2 \omega^{-2x_2} & -1 & \ddots  \\ 
    \vdots & \vdots & \vdots & \vdots & \ddots 
    \end{pmatrix}
\end{align*}
Lastly, we present $D_N$ in an expanded form and introduce new notation for the subsequent computation.
\begin{align} \label{det tridiag}
    \tilde{D}_N &:= \f1cD_N \notag \\
    &= 
    \det \begin{pmatrix}
    R_1 \omega^{-2x_1} & -1 & 0 & 0 & 0 & \hdots & 0 \\
    1 & -R_2 \omega^{2x_2} & -T_2 & 0 & 0 & \hdots & 0 \\
    0 & T_2 & R_2 \omega^{-2x_2} & -1 & 0 & \hdots & 0 \\
    0 & 0 & 1 & -R_3 \omega^{2x_3} & -T_3 & \hdots & 0 \\
    \vdots & \vdots & \vdots & \ddots & \ddots & \ddots & \vdots \\
    0 & 0 & 0 & \hdots & T_{N-1} & R_{N-1} \omega^{-2x_{N-1}} & -1 \\
    0 & 0 & 0 & \hdots & 0 & 1 & -R_N \omega^{2x_N}
    \end{pmatrix}
\end{align}
Further row reduction could turn this tridiagonal matrix symmetric, hence into a Jacobi matrix. We were unable to find an explicit formula suited to the goal of (eq \ref{lemma det eq}) in the literature, but the repeating entries in the main diagonal helps simplify it here. We set up a depth two recursion in the determinant (eq \ref{det tridiag}) and the same but one row/column smaller. The prior section simplified $\tilde{D}_N$ to coming from a $2N-2$ square matrix, while $d_{N-1}$ is a $2(N-1) - 1$ square matrix.
\begin{align}
    d_{N-1} &:= 
    \det \begin{pmatrix}
    R_1 \omega^{-2x_1}  & -1 &  &  & \\ % j=1
    1 & \ddots & \ddots& & \\ % j=2
    & \ddots & R_{N-2} \omega^{-2x_{N-2}} & -1 & \\
    & &  1 & -R_{N-1} \omega^{2x_{N-1}} & -T_{N-1} \\ % j=3
    & &    & T_{N-1} & R_{N-1} \omega^{-2x_{N-1}}
    \end{pmatrix}
\end{align}
In particular, cofactor expansion on the last column gives the following:
\begin{align*}
\tilde{D}_N = - R_N \omega^{2x_N} d_{N-1} + \tilde{D}_{N-1}
\end{align*}
The key step is the reduction of coefficients in $d_{N-1}$, following from that reciprocals of $\omega^{-2x_{N-1}}$ cancel and $T_j^2 - R_j^2 = (T_j - R_j)(T_j+R_j) = R_j+T_j$ simplifies.
\begin{align*}
d_{N-1} &= R_{N-1} \omega^{-2x_{N-1}} \tilde{D}_{N-1} + T_{N-1}^2 d_{N-2} \\
&= R_{N-1} \omega^{-2x_{N-1}} (-R_{N-1} \omega^{2x_{N-1}} d_{N-2} + \tilde{D}_{N-2}) + T_{N-1}^2 d_{N-2} \\
&= (R_{N-1}+T_{N-1}) d_{N-2} + R_{N-1} \omega^{-2x_{N-1}} \tilde{D}_{N-2}
\end{align*}
Iteratively plugging this into the recurrence for $\tilde{D}_N$ eliminates all $d$ terms at the expense of creating a full order recurrence in $\tilde{D}_N$.
\begin{align*}
    \tilde{D}_N &= -R_N \omega^{2x_N} \big( (R_{N-1}+T_{N-1}) d_{N-2} + R_{N-1} \omega^{-2x_{N-1}} \tilde{D}_{N-2}\big) + \tilde{D}_{N-1} \\
    &= \tilde{D}_{N-1} - R_N R_{N-1} \omega^{2(x_N-x_{N-1})} \tilde{D}_{N-2} - R_N (R_{N-1}+T_{N-1}) \omega^{2x_N} d_{N-2} \\
    &= \tilde{D}_{N-1} - R_N R_{N-1} \omega^{2(x_N-x_{N-1})} \tilde{D}_{N-2} - R_N (R_{N-1}+T_{N-1}) R_{N-2} \omega^{2(x_N-x_{N-2})} \tilde{D}_{N-3} \\
    &\quad - R_N (R_{N-1}+T_{N-1}) (R_{N-2}+T_{N-2}) \omega^{2x_N} d_{N-3}
\end{align*}
Then, we continue inductively until the base case $d_1 = R_1 \omega^{-2x_1}$.
\begin{align*}
    \tilde{D}_N &= \tilde{D}_{N-1} - R_{N-1}^N \tilde{D}_{N-2} - R_{N-2}^N \tilde{D}_{N-3} - ... - R_1^N \tilde{D}_0 \\
    &= \tilde{D}_{N-1} - \sum_{m=1}^{N-1} R_{N-m}^{N} \tilde{D}_{N-m-1}
\end{align*}
We see that the coefficient of $\tilde{D}_i$ only has $R_j$ and $x_j$ terms with $j > i$. Also, the base case of $\tilde{D}_2 = 1 - R_1 R_2 \omega^{2(x_2-x_1)}$ extrapolates to $\tilde{D}_1 = 1$ and $\tilde{D}_0 = 1$. We now perform induction to finish the proof. The base case is a simple $2 \times 2$ determinant, and the meaning of formula (eq \ref{lemma det eq}) is demonstrated in Example \ref{ex N2 det}.
\begin{align*}
    \tilde{D}_2 = \det \begin{pmatrix}
    R_1 \omega^{-2x_1} & -1 \\
    1 & -R_2 \omega^{2x_2}
    \end{pmatrix}
    = -R_1 R_2 \omega^{2(x_2-x_1)} + 1
\end{align*}
The step case follows from the recursion above. The key idea is that $\mathscr{I}_n^{N+1}$ can be decomposed into $\mathscr{I}_n^{N}$, which comes from $-\tilde{D}_{N-1}$, and $\mathscr{I}_n^{N+1} \big|_{k_n=N+1}$, which comes from the remaining $m=1,...,N-1$ terms.
\begin{align*}
    \tilde{D}_{N+1} &= \tilde{D}_N - \sum_{m=1}^{N} R_{N-m+1}^{N+1} \tilde{D}_{N-m} \\
    &= \tilde{D}_N - \sum_{m=1}^N R_{N-m+1}^{N+1} \* \left(1 + \sum_{n=1}^{\lfloor (N-m)/2\rfloor} \sum_{\mathscr{I}_n^{N-m}} (-1)^{n} R_{j_1}^{k_1} \* ... \* R_{j_n}^{k_n} \right) \\
    &= \tilde{D}_N - \sum_{m=1}^N \left(R_{N-m+1}^{N+1} + \sum_{n=1}^{\lfloor (N-m)/2\rfloor} \sum_{\mathscr{I}_n^{N-m}} (-1)^{n} R_{j_1}^{k_1} \* ... \* R_{j_n}^{k_n} \* R_{N-m+1}^{N+1} \right) \\
    &= \tilde{D}_N - \left( \sum_{\mathscr{I}_1^{N+1}}^{k_1=N+1} R_{j_1}^{N+1} + \sum_{n=2}^{\lfloor (N-1)/2\rfloor + 1} \sum_{\mathscr{I}_n^{N+1}}^{k_n=N+1} (-1)^{n-1} R_{j_1}^{k_1} \* ... \* R_{j_n}^{N+1} \right) \\
    &= \tilde{D}_N - \left( \sum_{n=1}^{\lfloor (N+1)/2\rfloor} \sum_{\mathscr{I}_n^{N+1}}^{k_n=N+1} (-1)^{n-1} R_{j_1}^{k_1} \* ... \* R_{j_n}^{N+1} \right) \\
    &= \left(1 + \sum_{n=1}^{\lfloor N/2\rfloor} \sum_{\mathscr{I}_n^N} (-1)^{n} R_{j_1}^{k_1} \* ... \* R_{j_n}^{k_n} \right)
    + \left( \sum_{n=1}^{\lfloor (N+1)/2\rfloor} \sum_{\mathscr{I}_n^{N+1}}^{k_n=N+1}  (-1)^{n} R_{j_1}^{k_1} \* ... \* R_{j_n}^{N+1} \right) \\
    &= 1 + \sum_{n=1}^{\lfloor (N+1)/2\rfloor} \sum_{\mathscr{I}_n^{N+1}} (-1)^n R_{j_1}^{k_1} \* ... \* R_{j_n}^{k_n} 
\end{align*}
\end{proof}
In further discussion, we motivated the reductions by bringing the matrix into tridiagonal form and the (eq \ref{RT notation}) notation as the most suitable abbreviations for this form. Alternatively, we could have isolated each $v_j^\pm$ and identified contributions from reflection at $x_j$ and transmission through $x_j$. However, $|R_j|^2$ and $|T_j|^2$ are not the reflection and transmission coefficients unless $N=1$. To derive (eq \ref{RT explanation}) from (eq \ref{det tridiag}), note that the vector $(v_0^-, v_1^-, v_1^+, v_2^-, v_2^+, ..., v_{N-1}^+, v_N^+)$ became $(v_1^-, v_1^+, v_2^-, v_2^+, ..., v_{N-1}^+)$ during the row reduction.
\begin{align} \label{RT explanation}
    v_{j-1}^- &= v_{j-1}^+ R_{j} \omega^{2x_j} + T_{j} v_{j}^- \\
    v_{j}^+ &= v_{j-1}^+ T_{j} + R_{j} \omega^{-2x_{j}} v_{j}^- \notag
\end{align}
\begin{figure}[H]
\centering
\begin{tikzpicture}
[scale=1, thick]
\draw[->, ultra thick] (2,0) -- (2,3);
\draw[->] (0,2) -- (1.9,2);
\draw[->] (0,2) -- (1.8,1.1);
\draw[->] (1.9,1) -- (0,1);
\draw[->] (4,1) -- (2.1,1);
\draw[->] (4,1) -- (2.2,1.9);
\draw[->] (2.1,2) -- (4,2);
\node[anchor=east] (v) at (0,2) {$v_{j-1}^+$};
\node[anchor=west] (v) at (4,2) {$v_j^+$};
\node[anchor=east] (v) at (0,1) {$v_{j-1}^-$};
\node[anchor=west] (v) at (4,1) {$v_j^-$};
\node[anchor=south] (v) at (1,2) {$T_j$};
\node[anchor=north] (v) at (3,1) {$T_j$};
\node[anchor=south west] (v) at (-0.25,1) {$R_j \omega^{2x_j}$};
\node[anchor=north east] (v) at (4.25,2) {$R_j \omega^{-2x_j}$};
\node[anchor=north] (v) at (2,0) {$x_j$};
\draw[<->, thin] (-0.5,0) -- (4.5,0);
\end{tikzpicture}
\caption{Refinement of Figure \ref{fig v^pm directions} with transmission and reflection at the boundary. It depicts (eq \ref{RT explanation}), where arrow heads to $v_j^+$, $v_{j-1}^-$ are outward travelling and arrow tails from $v_{j-1}^+$, $v_j^-$ are inward travelling to $x_j$.}
\end{figure}
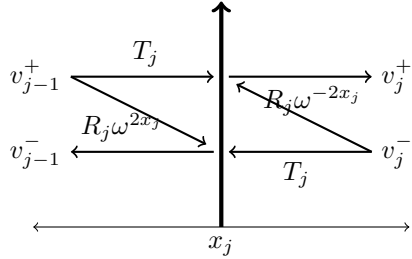

\section{Proof of Theorem \ref{thm structure}: motivating the Newton polygon}

Since the determinant formula is a complicated transcendental equation, the limit $h\rightarrow 0$ is key to our new results. The conditions in Theorem \ref{thm structure} helps simplify $z$ terms in the exponents and denominator. The Newton polygon then emerges by comparing the largest terms.

\begin{lemma} \label{lemma asymptotic}
    Let $z(h_j)$ be a sequence of resonances with $\Re z \in U \Subset (0,\infty)$ and $\Im z = -\gamma h\log(1/h) + O(h)$ as $h_j \rightarrow 0$. Then there exists $h_0$ such that for all $h\in (0,h_0)$, (eq \ref{lemma det eq}) simplifies to
    \begin{align} \label{eq lemma asymptotic det}
        \f1c D_N 
        = 1 + \sum_{1 \leq j< k \leq N} \Big(& \f{C_j C_k}{4z^2} \* h^{\beta_j + \beta_k} + o(h^{\beta_j + \beta_k})\Big) \omega^{2(x_k - x_j) }\text{,} \\
        % \text{with the order of magnitude given by} \qquad \qquad \qquad & \nonumber \\
        & \f{C_j C_k}{4z^2} \* h^{\beta_j + \beta_k} \omega^{2(x_k - x_j)} = O\big(h^{\beta_j + \beta_k -2(x_k - x_j) \gamma} \big) \text{,}
        \label{eq lemma asymptotic magnitude}
    \end{align}
    where we highlight the order of magnitude.
\end{lemma}
\begin{proof}
We used the $\omega, R_j, T_j$ notation in Lemma \ref{lemma det} to simplify the determinant, and now we unwind them to see the $h$ dependence. We take $\lambda$ below to be a positive term of lengths, namely $2(x_k - x_j)$ with $k>j$. The assumption on $\Im z$ is used to simplify $\omega$. A weaker, more elegant assumption would be $\Im z \sim -\gamma h\log(1/h)$, but we require more control on the smaller terms due to the exponential.
\begin{align} \label{eq omega simplified}
    \omega^{\lambda} &= e^{iz \lambda/h} 
    = e^{(-\Im z + i \Re z) \lambda/ h} 
    = e^{\gamma \lambda \log(1/h) + O(1)} e^{i (\Re z) \lambda / h}
    = h^{-\gamma \lambda} \* e^{i (\Re z) \lambda / h + O(1)} 
    = O(h^{-\gamma \lambda})
\end{align}
We simplify $R_j$ and $T_j$ as a geometric series by taking $h_0$ small enough and by using the assumption on $\Re z$ not too close to 0 to bound $1/z$ from above. The binomial expansion of $R_j^k$ then reduces since $T_j$'s are larger than $R_j$'s.
\begin{align} \label {wRT asymptotics}
    R_j &= \f{C_j h^{\beta_j}}{2iz - C_j h^{\beta_j}}
    = \f{C_j {h^\beta_j}}{2iz} \sum_{n=0}^\infty \left(\f{C_j h^{\beta_j}}{2iz}\right)^n 
    = \f{C_j h^{\beta_j}}{2iz} + O(h^{2\beta_j})
    \\
    T_j &= \f{2iz}{2iz - C_j h^{\beta_j}}
    = \sum_{n=0}^\infty \left(\f{C_j h^{\beta_j}}{2iz}\right)^n 
    = 1 + O(h^{\beta_j}) \\
    R_j & (T_{j+1}+R_{j+1}) \* ... \* (T_{k-1}+R_{k-1}) \* R_k
    = -\f{C_j C_k}{4z^2} h^{\beta_j+\beta_k} + o(h^{\beta_j+\beta_k})
\end{align}
The factor $-C_j C_k / 4z^2$ is $O(1)$ since $\Re z = O(1)$.
This gives the order of magnitude of terms in (eq \ref{eq lemma asymptotic magnitude}).
\begin{align*}
    R_{j}^{k} 
    = \Big( -\f{C_j C_k}{4z^2} \* h^{\beta_j + \beta_k} + o(h^{\beta_j + \beta_k})\Big) \omega^{2(x_k - x_j)}
    = O\Big( h^{\beta_{j} + \beta_{k} - 2 (x_{k}-x_{j}) \gamma} \Big)
\end{align*}
The last simplification is that the $n=1$ terms in (eq \ref{lemma det eq}) are largest. In particular, any term in $\mathscr{I}_n^N$ with $n>1$ has an associated term in $\mathscr{I}_1^N$ which is asymptotically larger, by dropping positive terms in the exponent.
\begin{align*}
    R_{j_1}^{k_1} \* ... \* R_{j_n}^{k_n}
    &= O\Big( h^{\beta_{j_1} + \beta_{k_1} + ... + \beta_{j_n} + \beta_{k_n} - \gamma (x_{k_1}-x_{j_1} + ... + x_{k_n}-x_{j_n})} \Big) \\
    &= o\Big( h^{\beta_{j_1} + \beta_{k_n} - \gamma (x_{k_n}-x_{j_1})} \Big) \\
    R_{j_1}^{k_n} &= O\Big( h^{\beta_{j_1} + \beta_{k_n} - \gamma (x_{k_n}-x_{j_1})} \Big)
\end{align*}
The result (eq \ref{eq lemma asymptotic det}) follows after combining these results:
\begin{align*}
    \f{1}{c} D_N &= 1 + \sum_{n=1}^{\lfloor N/2\rfloor} (-1)^{n} \sum_{\mathscr{I}_n^N} R_{j_1}^{k_1} \* ... \* R_{j_n}^{k_n} \\
    &= 1 + \sum_{\mathscr{I}_1^N} (-R_{j_1}^{k_1}) + \text{lower order terms} \\
    &= 1 + \sum_{1 \leq j< k \leq N} \Big( \f{C_j C_k}{4z^2} \* h^{\beta_j + \beta_k} + o(h^{\beta_j + \beta_k})\Big) \omega^{2(x_k - x_j) }
\end{align*}
Lastly, we can ignore the constant $c$ in $D_N = 0$ since $\Re z$ away from 0 guarantees $\f1c$ is bounded above.

\end{proof}

We now complete the proof of Theorem \ref{thm structure} and see how the Newton polygon emerges from Lemma \ref{lemma asymptotic}.

\begin{proof}[Proof of Theorem \ref{thm structure}]
Returning to the notation of Definition \ref{defn polygon}, let $\nu_i = \beta_j + \beta_k$, $\lambda_i = x_k - x_j$, and $c_i = O(1)$ for some indexing $i$. Then (eq \ref{eq lemma asymptotic det}) can be expressed as the following.
\begin{align} \label{det nu lambda}
    \f1c D_N = \sum_{i} \Big( c_i h^{\nu_i - \gamma \lambda_i} + o(h^{\nu_i - \gamma \lambda_i})\Big)
\end{align}
The insight from \cite{Datchev_Marzuola_Wunsch} is that there must be at least two commensurate terms of greatest magnitude for $D_N = 0$ to be satisfied. Asymptotically, this means two of the exponents of $h$ are equal and minimal.
\begin{align} \label{min con}
    \exists\: j,k \quad \nu_j - \gamma \lambda_j = \nu_k - \gamma \lambda_k = \min_{i} \nu_i - \gamma \lambda_i
\end{align}
After solving for $\gamma$, we see it is the slope between points $(\lambda_j,\nu_j)$ and $(\lambda_k,\nu_k)$ and express the minimality condition in a similar form. To simplify sign considerations, the cases $\lambda_i \geq \lambda_j$ and $\lambda_i \leq \lambda_k$ are equivalent by symmetry, so assume the former without loss of generality. Also assuming $\lambda_j < \lambda_k$, we have that $\lambda_k-\lambda_j$ and $\lambda_i - \lambda_j$ are both positive.
\begin{align*}
    \gamma = \f{\nu_k-\nu_j}{\lambda_k-\lambda_j} \leq \f{\nu_i - \nu_j}{\lambda_i - \lambda_j}
\end{align*}
We will visualize the possibilities of $\gamma$ with the Newton polygon. The minimality condition has the geometric interpretation that the segment $(\lambda_j,\nu_j)$ to $(\lambda_k,\nu_k)$ is on the Newton polygon of the points $(\lambda_i, \nu_i)$. This is because lines on the lower convex hull have smaller slope than lines into the interior, and the ordering $\lambda_j \leq \lambda_i$ ensures the slope is defined in the natural direction.
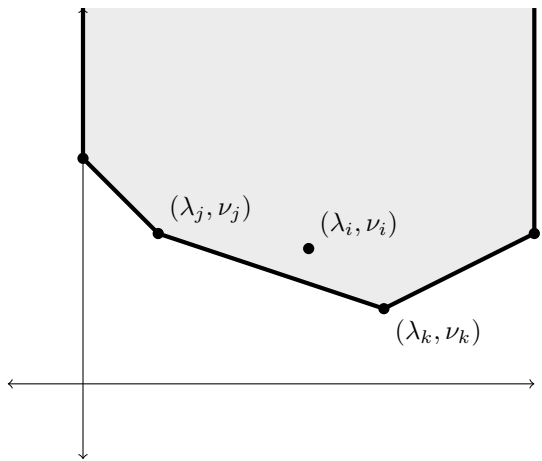
\begin{figure}[H]
\centering
\begin{tikzpicture}
    [scale=1, ultra thick]
     \draw[fill=gray!15] (0,5) -- (0,3) -- (1,2) -- (4,1) -- (6,2) -- (6,5);
    \draw[<->, thin] (0,-1) -- (0,5);
    \draw[<->, thin] (-1,0) -- (6,0);
    \node[anchor=north west] (bT) at (4,1) {$(\lambda_k,\nu_k)$};
    \node[anchor=south west] (bT) at (1,2) {$(\lambda_j,\nu_j)$};
    \node[anchor=south west] (bT) at (3,1.8) {$(\lambda_i,\nu_i)$};
    \node[circle, fill, inner sep=1.5pt] (test) at (0,3) {};
    \node[circle, fill, inner sep=1.5pt] (test) at (1,2) {};
    \node[circle, fill, inner sep=1.5pt] (test) at (3,1.8) {};
    \node[circle, fill, inner sep=1.5pt] (test) at (4,1) {};
    \node[circle, fill, inner sep=1.5pt] (test) at (6,2) {};
    \node[anchor=north west] (g) at (0,5) {$\nu$};
    \node[anchor=south east] (g) at (6,0) {$\lambda$};
\end{tikzpicture}
\caption{A Newton polygon, as defined in Definition \ref{defn polygon} and more general than the specific Newton polygon in Definition \ref{defn our polygon}. Note that the slope from points indexed $j$ to $k$ is smaller than the slope from $j$ to $i$ on the right and in the interior.}
\end{figure}
Lastly, we note that $\nu = \beta_j+\beta_k$ and $\lambda = 2(x_k-x_j)$ over $1\leq j < k \leq N$, with $(0,0)$ for the 1 term, is indeed our particular Newton polygon in Definition \ref{defn our polygon}. Hence, under the conditions of Lemma \ref{lemma asymptotic}, we have $\gamma \in \Gamma$.
\end{proof}

We have now connected resonances of logarithmic form to slopes of the Newton polygon, but the discussion above says more about the structure in Lemma \ref{lemma asymptotic}. We return to it during the proof of Theorem \ref{thm existence} with the following link, which is our most significant use of Assumption \ref{assumption}.

\begin{lemma} \label{lemma polygon}
    For any $\gamma^* \in \Gamma$, denote $\nu_j - \gamma^* \lambda_j$ and $\nu_k - \gamma^* \lambda_k$ as the equal and strictly minimum exponents in (eq \ref{min con}). There exists $\eps$ such that the exponents $\nu_j - \gamma \lambda_j$ and $\nu_k - \gamma \lambda_k$ remain minimal, though not necessarily equal, for all $\gamma$ near $\gamma^*$ as follows:
    \begin{align}
        \gamma \in [\gamma^*-\eps, \gamma^*+\eps]
    \end{align}
    Let $z(h_j)$ be a sequence of resonances with $\Re z \in U \Subset (0,\infty)$ and $\Im z = -\gamma h\log(1/h) + O(h)$ as $h_j \rightarrow 0$. Then the largest terms in (eq \ref{eq lemma asymptotic det}) are the two corresponding to $(\nu_j,\lambda_j)$ and $(\nu_k,\lambda_k)$.
\end{lemma}
\begin{proof}
Figure \ref{fig distorted polygon} is central to the proof. It depicts the exponents $\beta_j+\beta_k - 2\gamma(x_k-x_j)$ as a function of $\gamma$ to clarify how we find the largest term for any $\gamma$, even though $\gamma \in \Gamma$ is a necessary condition for $z(h_j)$ to be resonances. Assumption \ref{assumption} guarantees there are exactly two vertices $(\lambda,\nu)$ on each polygon line segment of $\gamma^*$, the exponents $\nu_j-\gamma^*\lambda_j$ and $\nu_k - \gamma^* \lambda_k$ are strictly minimal, and there are not three teal lines through one blue dot in Figure \ref{fig distorted polygon}. This justifies the figure and its use in the proof.
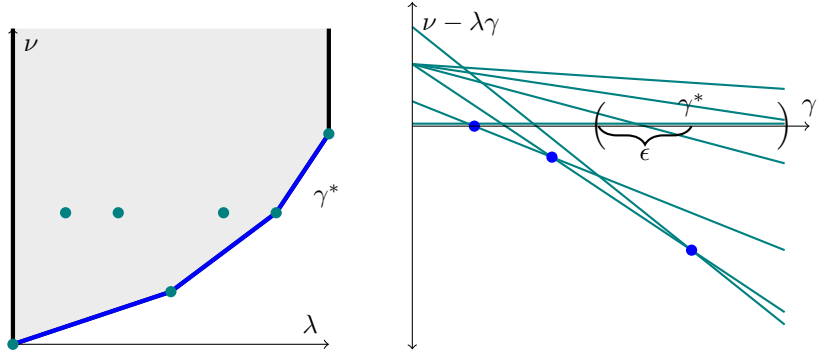
\begin{figure}[H]
    \centering
\begin{tikzpicture} % Newton polygon
    [scale=.7, ultra thick]
    \draw[fill=gray!15] (0,6) -- (0,0) -- (3,1) -- (5,2.5) -- (6,4) -- (6,6);
    \draw[blue] (0,0) -- (3,1) -- (5,2.5) -- (6,4);
    \draw[->, thin] (0,0) -- (0,6);
    \draw[->, thin] (0,0) -- (6,0);
    \node[anchor=north west] (g) at (0,6) {$\nu$};
    \node[anchor=south east] (g) at (6,0) {$\lambda$};
    \node[circle, fill, inner sep=1.5pt, teal] (point) at (0,0) {};
    \node[circle, fill, inner sep=1.5pt, teal] (point) at (1,2.5) {};
    \node[circle, fill, inner sep=1.5pt, teal] (point) at (2,2.5) {};
    \node[circle, fill, inner sep=1.5pt, teal] (point) at (3,1.0) {};
    \node[circle, fill, inner sep=1.5pt, teal] (point) at (4,2.5) {};
    \node[circle, fill, inner sep=1.5pt, teal] (point) at (5,2.5) {};
    \node[circle, fill, inner sep=1.5pt, teal] (point) at (6,4.0) {};
    \node[anchor=north west] (g) at (5.5,3.25) {$\gamma^*$};
\end{tikzpicture} \qquad
\begin{tikzpicture}
    [scale=0.33, thick]
    \draw[teal] (0,4) -- (15,-8);
    \draw[teal] (0,2.5) -- (15,-7.5);
    \draw[teal] (0,2.5) -- (15,1.5);
    \draw[teal] (0,2.5) -- (15,.25);
    \draw[teal] (0,2.5) -- (15,-1.5);
    \draw[teal] (0,1) -- (15,-5);
    \draw[teal] (0,0.1) -- (15,0.1);
    \node[circle, fill, inner sep=1.5pt, blue] (point) at (15/6,0) {};
    \node[circle, fill, inner sep=1.5pt, blue] (point) at (45/8,-1.25) {};
    \node[circle, fill, inner sep=1.5pt, blue] (point) at (45/4,-5) {};
    \node[anchor=north west] (g) at (0,5) {$\nu-\lambda \gamma$};
    \node[anchor=south] (g) at (16,0) {$\gamma$};
    \node[anchor=south] (g) at (45/4,0) {$\gamma^*$};
    \node[anchor=north] (g) at (37.5/4,-0.5) {$\eps$};
    \node (g) at (15/2,0) {$\Big($};
    \node (g) at (15,0) {$\Big)$};
    \draw[decorate, decoration = {brace, amplitude=7pt}] (45/4,0) -- (15/2,0);
    \draw[<->, thin] (0,5) -- (0,-9);
    \draw[->, thin] (0,0) -- (16,0);
\end{tikzpicture}
    \caption{A Newton polygon compared to the exponents of $h$ graphed against $\gamma$. Points (teal) in and on the Newton polygon are lines in this new graph, and slopes of $\gamma$ (blue) are the intersections of the lower boundary of these lines. We also display the interval $[\gamma^*-\eps, \gamma^*+\eps]$ guaranteed by Lemma \ref{lemma polygon} applied to the third slope of the polygon. The underlying parameters are from Example \ref{ex N4} case $c$.}
    \label{fig distorted polygon}
\end{figure}
The existence of $\eps$ that maintains the minimality of $\nu_j-\gamma \lambda_j$ and $\nu_k-\gamma \lambda_k$ follows from taking the minimum over a finite set. Figure \ref{fig distorted polygon} shows how the $\eps$ window extends to the intersection of teal lines nearest to the blue dot of $\gamma^*$. Lastly, the association between the magnitude of terms in (eq \ref{eq lemma asymptotic det}) and $(\nu_j,\lambda_j)$, $(\nu_k,\lambda_k)$ was demonstrated in the proof of Theorem \ref{thm structure} above.
\end{proof}

\section{Proof of Theorem \ref{thm Gamma}: characterization of the Newton polygon}

This section characterizes the slopes on the Newton polygon, independent of other meaning in the parameters. The first line, from left to right, is easiest since it is the smallest slope of $(0,0)$ to every point. As discussed around Definition \ref{defn dominant pair}, these slopes are $\gamma_{jk}$ in (eq \ref{eq gamma defn}) and this minimization is (eq \ref{min con}). This justifies the dominant string.
\begin{align*}
    \gamma_{JK} \in \Gamma
\end{align*}
However, the form of the remaining slopes are not, a priori, $\tilde{\gamma}_{jk}$. Most generally, the slopes are between two points $\big(2|x_j-x_i|,\beta_i+\beta_j\big)$ and $\big(2|x_k-x_l|,\beta_l+\beta_k\big)$, for some order of $j,k,i,l$. We show below that in fact, adjacent points on the convex hull have one index in common, $i = l$. If this is before $j$ and $k$, we can write:
\begin{align*}
    \f{(\beta_l+\beta_k)-(\beta_i+\beta_j)}{2(x_k-x_l) - 2(x_j-x_i)}
    = \f{\beta_k-\beta_j+\beta_l-\beta_i}{2(x_k-x_j) + 2(x_i-x_l)}
    = \f{\beta_k-\beta_j}{2(x_k-x_j)}
    = \pm \tilde{\gamma}_{jk}
\end{align*}
Our proof uses a geometric interpretation of vertices $(\lambda_i,\nu_i)$ on the polygon, that each vertex $V=V(j,k)$ corresponds to the pair of delta functions $j<k$ and hence an interval $[x_j,x_k]$. This is a possibly confusing distinction to our theorem, where each slope also corresponds to a pair of delta functions. By showing adjacent vertices on the Newton polygon have a delta function in common, and we will find that the delta functions associated to the slope between them are the two not in common. This is described in Figure \ref{fig vertices vs sides}. \\
\begin{figure}[H]
\label{fig vertices vs sides}
\centering
\resizebox{0.85\textwidth}{!}{
\begin{tikzpicture}
[scale=1.0, ultra thick]
\draw[->] (0,0) -- (0,0.5);
\draw[->] (2,0) -- (2,2);
\draw[->] (5,0) -- (5,2);
\draw[->] (6,0) -- (6,0.5);
\node[anchor=south] (b1) at (0,0.5) {$\beta_1$};
\node[anchor=south] (b2) at (2,2) {$\beta_2$};
\node[anchor=south] (b3) at (5,2) {$\beta_3$};
\node[anchor=south] (b4) at (6,0.5) {$\beta_4$};
\draw[decorate, decoration = {brace, amplitude=10pt}] (5,-0.2) -- (2,-0.2);
\draw[decorate, decoration = {brace, amplitude=3pt}] (2,-0.2) -- (0,-0.2);
\draw[decorate, decoration = {brace, amplitude=3pt}] (6,-0.2) -- (5,-0.2);
\node (g) at (1,-0.6) {$\tilde{\gamma}_{12}$};
\node (g) at (3.5,-0.8) {$\gamma_{23}$};
\node (g) at (5.5,-0.6) {$\tilde{\gamma}_{34}$};
\draw[decorate, decoration = {brace, amplitude=3pt}] (2,2.5) -- (5,2.5);
\draw[decorate, decoration = {brace, amplitude=7pt, aspect=1/5}] (0,2.7) -- (5,2.7);
\draw[decorate, decoration = {brace, amplitude=10pt, aspect=5.5/6}] (0,2.9) -- (6,2.9);
\node[anchor=south] (g) at (1,3.2) {$V(1,3)$};
\node[anchor=south] (g) at (3.5,3.2) {$V(2,3)$};
\node[anchor=south] (g) at (5.5,3.2) {$V(1,4)$};
\draw[<->, thin] (-0.5,0) -- (6.5,0);
\node (placement) at (0,-1) { };
\end{tikzpicture} \;
\begin{tikzpicture}
    [scale=1, ultra thick]
    \draw[fill=gray!15] (0,5) -- (0,0) -- (3,1) -- (5,2.5) -- (6,4) -- (6,5);
    \draw[->, thin] (0,0) -- (0,5);
    \draw[->, thin] (0,0) -- (6,0);
    \node[anchor=south east] (bT) at (3,1) {$V(2,3)$};
    \node[anchor=south east] (cT) at (5,2.5) {$V(1,3)$};
    \node[anchor=south east] (dT) at (6,4.0) {$V(1,4)$};
    \node[anchor=north west] (g) at (1.5,0.5) {$\gamma_{23}$};
    \node[anchor=north west] (g) at (4,1.75) {$\tilde{\gamma}_{12}$};
    \node[anchor=north west] (g) at (5.5,3.25) {$\tilde{\gamma}_{34}$};
    \node[circle, fill, inner sep=1.5pt] (test) at (1,2.5) {};
    \node[circle, fill, inner sep=1.5pt] (test) at (2,2.5) {};
    \node[circle, fill, inner sep=1.5pt] (test) at (4,2.5) {};
    \node[circle, fill, inner sep=1.5pt] (test) at (0,0) {};
    \node[circle, fill, inner sep=1.5pt] (test) at (3,1) {};
    \node[circle, fill, inner sep=1.5pt] (test) at (5,2.5) {};
    \node[circle, fill, inner sep=1.5pt] (test) at (6,4.0) {};
    \node[anchor=north] (axes) at (3,0) {$\lambda = 2(x_k - x_j)$};
    \node[anchor=south, rotate=90] (axes) at (0,2.5) {$\nu = \beta_j + \beta_k$};
\end{tikzpicture}
} % end resize box
\caption{The distinction between vertices and sides on the Newton polygon. We show that adjacent vertices correspond to pairs of delta functions in nested intervals and hence adjacent sides correspond to delta functions in adjacent intervals. These delta functions are denoted in generality but are based on Example \ref{ex N4} part $c$.}
\end{figure}
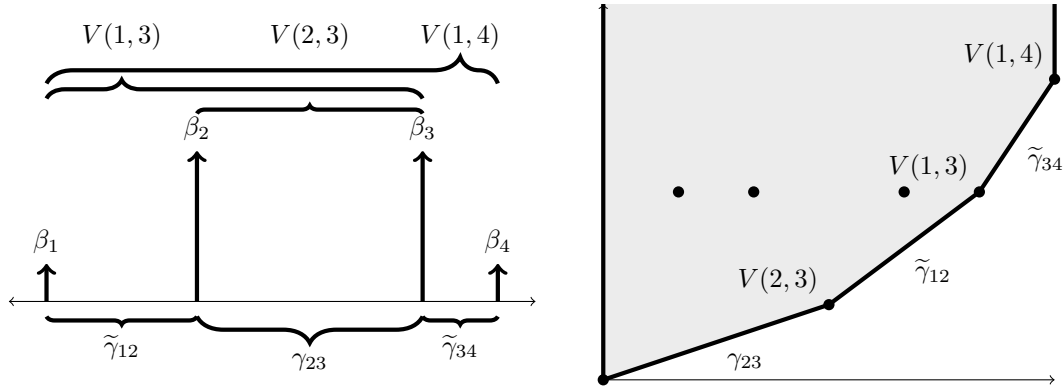
\begin{proposition} \label{prop polygon intervals}
    Recall that vertices forming the Newton polygon are defined by a pair of delta functions. Then delta function pairs on the polygon must correspond to a unique set of nested intervals with one endpoint in common.
\end{proposition}
\begin{proof}
We eliminate the following structures of delta functions pairs on the polygon, leaving only the desired arrangement possible. The graphics are to scale for specific parameters but represent arbitrary delta functions.\par
Our tool for this is the fact that if $\lambda_i \leq \lambda_j$ and $\nu_i \geq \nu_j$, then $(\lambda_i,\nu_i)$ cannot be a vertex on the polygon unless $(\lambda_i,\nu_i) = (\lambda_j,\nu_j)$. In other words, points with smaller distance values and greater $\beta$ values are overshadowed by the other point. This is because $(0,0)$ is a point in (eq \ref{lemma det eq}) and $\nu_i, \lambda_i \geq 0$ for all $i$.
\begin{enumerate}[label=\alph*]
    \item % a
\textit{Delta function pairs on the polygon cannot be in adjacent, disconnected intervals.} Let $V(1,2)$ and $V(3,4)$ be these adjacent, disconnected delta function pairs. For $V(1,2)$ to not be overshadowed by $V(1,3)$ on the polygon, the second delta must be stronger than the third, i.e. $\beta_2 < \beta_3$. For $V(3,4)$ to not be overshadowed by $V(2,4)$, we similarly must have $\beta_3 < \beta_2$. These inequalities contradict each other, hence we cannot have both $V(1,2)$ and $V(3,4)$ on the polygon.
\begin{figure}[H]
\centering
\begin{tikzpicture}
    [scale=0.7, ultra thick]
    \draw[->] (0,0) -- (0,0.5);
    \draw[->] (2,0) -- (2,0.5);
    \draw[->] (5,0) -- (5,2.0);
    \draw[->] (6,0) -- (6,0.5);
    \node[anchor=south] (b1) at (0,0.5) {$\beta_1$};
    \node[anchor=south] (b2) at (2,0.5) {$\beta_2$};
    \node[anchor=south] (b3) at (5,2.0) {$\beta_3$};
    \node[anchor=south] (b4) at (6,0.5) {$\beta_4$};
    \draw[decorate, decoration = {brace, amplitude=7pt, aspect=1/2}] (0,3) -- (2,3);
    \draw[decorate, decoration = {brace, amplitude=7pt, aspect=1/2}] (5,3) -- (6,3);
    \node[anchor=south] (g) at (1,3.5) {$V(1,2)$};
    \node[anchor=south] (g) at (5.5,3.5) {$V(3,4)$};
    \draw[<->, thin] (-0.5,0) -- (6.5,0);
\end{tikzpicture} \qquad
\begin{tikzpicture}
    [scale=0.7, ultra thick]
    \draw[fill=gray!15] (0,5) -- (0,0) -- (5,2.5) -- (6,4) -- (6,5);
    \draw (1,2.5) -- (4,4);
    \draw (2,4) -- (5,2.5);
    \draw[->, thin] (0,0) -- (0,5);
    \draw[->, thin] (0,0) -- (6,0);
    \node[anchor=north] (bT) at (1,2.5) {$V(3,4)$};
    \node[anchor=south] (cT) at (2,4) {$V(1,2)$};
    \node[anchor=south] (dT) at (4,4) {$V(2,4)$};
    \node[anchor=north] (dT) at (5,2.5) {$V(1,3)$};
    \node[circle, fill, inner sep=1.5pt] (test) at (1,2.5) {};
    \node[circle, fill, inner sep=1.5pt] (test) at (2,4) {};
    \node[circle, fill, inner sep=1.5pt] (test) at (4,4) {};
    \node[circle, fill, inner sep=1.5pt] (test) at (0,0) {};
    \node[circle, fill, inner sep=1.5pt] (test) at (3,2.5) {};
    \node[circle, fill, inner sep=1.5pt] (test) at (5,2.5) {};
    \node[circle, fill, inner sep=1.5pt] (test) at (6,4.0) {};
\end{tikzpicture}
\end{figure}
    \item % b
\textit{Delta function pairs on the polygon cannot be in adjacent,  connected intervals.} The argument above does not work since the inequalities $\beta_2 < \beta_3$ and $\beta_3 < \beta_2$ are no longer strict. Instead, we use that the slope from $(0,0)$ to $V(2,3)$ of $(\beta_2+\beta_3)/(x_3-x_2)$ is greater than the slope from $V(1,2)$ to $V(1,3)$ of $(\beta_3-\beta_2)/(x_3-x_2)$, which follows directly from $\beta_2 > 0$. Since slopes on the polygon must be increasing, we cannot have both $V(1,2)$ and $V(2,3)$ on the polygon.
\begin{figure}[H]
\centering
\begin{tikzpicture}
    [scale=0.7, ultra thick]
    \draw[->] (0,0) -- (0,0.5);
    \draw[->] (2,0) -- (2,2);
    \draw[->] (5,0) -- (5,2);
    \draw[->] (6,0) -- (6,0.5);
    \node[anchor=south] (b1) at (0,0.5) {$\beta_1$};
    \node[anchor=south] (b2) at (2,2) {$\beta_2$};
    \node[anchor=south] (b3) at (5,2) {$\beta_3$};
    \node[anchor=south] (b4) at (6,0.5) {$\beta_4$};
    \draw[decorate, decoration = {brace, amplitude=7pt, aspect=1/2}] (0,3) -- (2,3);
    \draw[decorate, decoration = {brace, amplitude=7pt, aspect=2/3}] (2,3) -- (5,3);
    \node[anchor=south] (g) at (1,3.5) {$V(1,2)$};
    \node[anchor=south] (g) at (3.5,3.5) {$V(2,3)$};
    \draw[<->, thin] (-0.5,0) -- (6.5,0);
\end{tikzpicture} \qquad
\begin{tikzpicture}
    [scale=0.7, ultra thick]
    \draw[fill=gray!15] (0,5) -- (0,0) -- (3,1) -- (5,2.5) -- (6,4) -- (6,5);
    \draw (0,0) -- (2,2.5) -- (5,2.5);
    \draw[->, thin] (0,0) -- (0,5);
    \draw[->, thin] (0,0) -- (6,0);
    \node[anchor=south east] (bT) at (3,1) {$V(2,3)$};
    \node[anchor=south east] (cT) at (5,2.5) {$V(1,3)$};
    \node[anchor=south east] (dT) at (2,2.5) {$V(1,2)$};
    \node[circle, fill, inner sep=1.5pt] (test) at (2,2.5) {};
    \node[circle, fill, inner sep=1.5pt] (test) at (0,0) {};
    \node[circle, fill, inner sep=1.5pt] (test) at (3,1) {};
    \node[circle, fill, inner sep=1.5pt] (test) at (5,2.5) {};
    \node[circle, fill, inner sep=1.5pt] (test) at (6,4.0) {};
\end{tikzpicture}
\end{figure}
    \item % c
\textit{Delta function pairs on the polygon cannot be in overlapping intervals.} Let $V(1,3)$ and $V(2,4)$ be these overlapping pairs. The lines $V(2,3)$ to $V(1,3)$ and $V(2,4)$ to $V(1,4)$ have the same slope of $(\beta_2-\beta_1)/(x_2-x_1)$ because they both trade out delta functions 2 for 1. These four points make a parallelogram with $V(2,3)$ is furthest left and $V(1,4)$ is furthest right. Since parallelograms are convex, we cannot have the remaining corners $V(1,3)$ and $V(2,4)$ simultaneously on the polygon. Also, Assumption \ref{assumption} guarantees that the parallelogram is non-degenerate.
\begin{figure}[H]
\centering
\begin{tikzpicture}
    [scale=0.7, ultra thick]
    \draw[->] (0,0) -- (0,0.5);
    \draw[->] (2,0) -- (2,2);
    \draw[->] (5,0) -- (5,2);
    \draw[->] (6,0) -- (6,0.5);
    \node[anchor=south] (b1) at (0,0.5) {$\beta_1$};
    \node[anchor=south] (b2) at (2,2) {$\beta_2$};
    \node[anchor=south] (b3) at (5,2) {$\beta_3$};
    \node[anchor=south] (b4) at (6,0.5) {$\beta_4$};
    \draw[decorate, decoration = {brace, amplitude=7pt, aspect=1/5}] (0,3.0) -- (5,3.0);
    \draw[decorate, decoration = {brace, amplitude=10pt, aspect=5.5/6}] (2,3.2) -- (6,3.2);
    \node[anchor=south] (g) at (1,3.7) {$V(1,3)$};
    \node[anchor=south] (g) at (5.5,3.7) {$V(2,4)$};
    \draw[<->, thin] (-0.5,0) -- (6.5,0);
\end{tikzpicture} \qquad
\begin{tikzpicture}
    [scale=0.7, ultra thick]
    \draw[fill=gray!15] (0,5) -- (0,0) -- (3,1) -- (5,2.5) -- (6,4) -- (6,5);
    \draw (3,1) -- (4,2.5) -- (6,4.0);
    \draw[->, thin] (0,0) -- (0,5);
    \draw[->, thin] (0,0) -- (6,0);
    \node[anchor=south east] (bT) at (3,1) {$V(2,3)$};
    \node[anchor=east] (cT) at (4,2.5) {$V(2,4)$};
    \node[anchor=west] (cT) at (5,2.5) {$V(1,3)$};
    \node[anchor=south east] (dT) at (6,4.0) {$V(1,4)$};
    \node[circle, fill, inner sep=1.5pt] (test) at (4,2.5) {};
    \node[circle, fill, inner sep=1.5pt] (test) at (0,0) {};
    \node[circle, fill, inner sep=1.5pt] (test) at (3,1) {};
    \node[circle, fill, inner sep=1.5pt] (test) at (5,2.5) {};
    \node[circle, fill, inner sep=1.5pt] (test) at (6,4.0) {};
\end{tikzpicture}
\end{figure}
    \item % d
\textit{Delta function pairs on the polygon cannot be in strictly nested intervals.} This situation would be $V(2,3)$ and $V(1,4)$ appearing the polygon without either $V(1,3)$ or $V(2,4)$ on the polygon. The same picture as above suggests this holds due to the convexity of the parallelogram, with Assumption \ref{assumption} providing the non-degeneracy condition. Also, we cannot have another point $V(5,6)$ between $V(2,3)$ and $V(1,4)$ which overshadows both $V(1,3)$ and $V(2,4)$, or we can iterate the argument and use that $N$ is finite.
\begin{figure}[H]
\centering
\begin{tikzpicture}
    [scale=0.7, ultra thick]
    \draw[->] (0,0) -- (0,0.5);
    \draw[->] (2,0) -- (2,2);
    \draw[->] (5,0) -- (5,2);
    \draw[->] (6,0) -- (6,0.5);
    \node[anchor=south] (b1) at (0,0.5) {$\beta_1$};
    \node[anchor=south] (b2) at (2,2) {$\beta_2$};
    \node[anchor=south] (b3) at (5,2) {$\beta_3$};
    \node[anchor=south] (b4) at (6,0.5) {$\beta_4$};
    \draw[decorate, decoration = {brace, amplitude=3pt, aspect=1/3}] (2,3) -- (5,3);
    \draw[decorate, decoration = {brace, amplitude=10pt, aspect=5/6}] (0,3.2) -- (6,3.2);
    \node[anchor=south] (g) at (3,3.5) {$V(2,3)$};
    \node[anchor=south] (g) at (5,3.5) {$V(1,4)$};
    \draw[<->, thin] (-0.5,0) -- (6.5,0);
\end{tikzpicture} \qquad
\begin{tikzpicture}
    [scale=0.7, ultra thick]
    \draw[fill=gray!15] (0,5) -- (0,0) -- (3,1) -- (5,2.5) -- (6,4) -- (6,5);
    \draw (3,1) -- (4,2.5) -- (6,4.0);
    \draw[->, thin] (0,0) -- (0,5);
    \draw[->, thin] (0,0) -- (6,0);
    \node[anchor=south east] (bT) at (3,1) {$V(2,3)$};
    \node[anchor=east] (cT) at (4,2.5) {$V(2,4)$};
    \node[anchor=west] (cT) at (5,2.5) {$V(1,3)$};
    \node[anchor=south east] (dT) at (6,4.0) {$V(1,4)$};
    \node[circle, fill, inner sep=1.5pt] (test) at (4,2.5) {};
    \node[circle, fill, inner sep=1.5pt] (test) at (0,0) {};
    \node[circle, fill, inner sep=1.5pt] (test) at (3,1) {};
    \node[circle, fill, inner sep=1.5pt] (test) at (5,2.5) {};
    \node[circle, fill, inner sep=1.5pt] (test) at (6,4.0) {};
\end{tikzpicture}
\end{figure}
\end{enumerate}
Parts $a$-$c$ show that the intervals must be nested, and part $d$ shows one endpoint of consecutive intervals must be in common, proving the claim.
\end{proof}
This lemma essentially completes the proof of Theorem \ref{thm Gamma}, which we tie together now.
\begin{proof}[Proof of Theorem \ref{thm Gamma}]
The discussion preceding Proposition \ref{prop polygon intervals} characterizes the form of slopes on the polygon into $\gamma_{jk}$ and $\tilde{\gamma}_{jk}$. The absolute values in (eq \ref{eq gamma defn}) are a notational convenience to avoid casework in the location of $i=l$ relative to $j$ and $k$. Nevertheless, this casework shows that delta functions contributing vertices to Newton polygon must become stronger ($\beta_j<\beta_k$) if $j$ is closer to the dominant pair. \par
Proposition \ref{prop polygon intervals} itself gives the partitioning $1=j_1 < ... <j_n = N$. This partition covers $[1,N]$ since $V(1,N)$ is always on the polygon, which is a result of $x_N-x_1$ being the largest length between delta functions.
\end{proof}
Note that Proposition \ref{prop polygon intervals} requires Assumption \ref{assumption}, but we claim without proof that Theorem \ref{thm Gamma} holds without it. In particular, the ways to choose intervals in Proposition \ref{prop polygon intervals} becomes non-unique, but the values of $\tilde{\gamma}_{jk}$ claimed in Theorem \ref{thm Gamma} remain in $\Gamma$.

\section{Proof of Theorem \ref{thm existence}: existence by Rouch\'e's theorem}

With the conditions in Theorem \ref{thm structure}, we have found the precise $\gamma$ for which we expect it to hold. We now use the properties of these $\gamma$ to show there are resonances following these conditions. Recall that Rouche's theorem states the number of roots of a function $f(z)$ determine the number of roots of a more complicated function $f(z) + g(z)$ as long as $|g(z)| < |f(z)|$ on the boundary of the simply connected region of interest and both $f(z),g(z)$ are analytic.

We first prove the theorem for the case of $N=2$, where we need consider only one $\gamma$. This was proved in \cite{Datchev_Marzuola_Wunsch} with essentially the same method, but we elaborate on the argument here as a preliminary to the more general case. We use the limit $h \rightarrow 0$ not only to bound $|g(z)| < |f(z)|$ but also to show $g(z)$ is analytic, hence we emphasize the order of these limits.

\begin{proposition} \label{prop N2}
For $N=2$, we have that for any $M$ large, there exists $h_0$ small such that for all $h \in (0,h_0)$, any integer $m$ satisfying $\pi h m / (x_k-x_j) \in [\f{2}{M},\f{M}{2}]$ gives a resonance of the following form. Furthermore, all resonances in the interior of the lower right quadrant lie on this line.
\begin{align} \label{prop N2 eq}
    z_m &= \f{\pi h m}{x_2-x_1} - i \gamma_{12} h \log(1/h) + O(h)
\end{align}
\end{proposition}
\begin{proof}
We start from the determinant formula for $z$ in Example \ref{ex N2 det} and attempt to isolate $z$. The last equality uses the leading terms of $R_1 R_2$ identified in (eq \ref{wRT asymptotics}). Since $m \in \mathbb{Z}$ is arbitrary, we redefine $m$ as $-m$ for brevity.
\begin{align*}
    e^{-2izl/h} &= \f{C_1 h^\beta_1}{2iz - C_1 h^\beta_1} \f{C_2 h^\beta_2}{2iz - C_2 h^\beta_2} \\
    -2izl/h &= 2\pi i m - (\beta_1 + \beta_2) \log(1/h) + \log \left(\f{C_1}{2iz - C_1 h^\beta_1} \f{C_2}{2iz - C_2 h^\beta_2}\right) \\ 
    z &= \f{\pi h (-m)}{l} - i \f{\beta_1 + \beta_2}{2l} h\log(1/h) + \f{ih}{2l} \log \left(\f{C_1}{2iz - C_1 h^\beta_1} \f{C_2}{2iz - C_2 h^\beta_2}\right) \\
    z &= \f{\pi h m}{l} - i \f{\beta_1+\beta_2}{2l} h \log(1/h) + \f{ih}{2l} \log \left( \f{C_1 C_2}{-4z^2} (1+\tilde{R})\right)
\end{align*}
Our approach is to use Rouche's theorem on $f(z)$, $g(z)$, and $E \subset \mathbb{C}$ as follows.
\begin{align*}
    f(z) &:= \f{\pi h m}{l} - i\gamma_{12} h\log(1/h) - z \\
    g(z) &:= \f{ih}{2l} \log \left( \f{C_1 C_2}{-4z^2} (1+\tilde{R})\right) \\
    E &:= [1/M, M] + i [-M, 0]
\end{align*}

\begin{figure}[H]
    \centering
\begin{tikzpicture}
    [scale=0.8, ultra thick]
    \draw[fill=gray!15] (.2,0) -- (5,0) -- (5,-5) -- (.2,-5) --(.2,0);
    \draw[->] (0.5,-1.5) -- (2.5,-1.5);
    \node[anchor=east] (g1) at (0, -1.5) {$-i\gamma_{12} h\log(1/h)$};
    \node[anchor=east] (g1) at (0, -5) {$-iM$};
    \node[anchor=south] (g1) at (5, 0) {$M$};
    \node[anchor=south] (g1) at (0.2, 0) {$\f{1}{M}$};
    \node[] (g1) at (2.5, -2.5) {$E$};
    \draw[->, thin] (0,0) -- (5.5,0);
    \draw[->, thin] (0,0) -- (0,-5.5);
    \node[anchor=south] (axes) at (2.5,0) {$\Re z$};
    \node[anchor=south, rotate=90] (axes) at (0,-2.5) {$\Im z$};
\end{tikzpicture}
\end{figure}
We carefully describe the order of limits.
\begin{enumerate}
    \item Let $M >> 1$ be a large, positive number, making $E$ fill the lower right quadrant. The real part of $z$ then gives $z = O(1)$ with respect to $h$.
    \item Choose $h_0$ such that $|g(z)| < \gamma_{12} h\log(1/h) < 1/M$ and $|\tilde{R}| << 1$ for all $h\in (0,h_0)$. The former is possible since $h < h\log(1/h) \rightarrow 0$ as $h\rightarrow 0$, and the latter follows from a geometric series argument on $1/(1-\f{C_jh^{\beta_j}}{2iz})$.
    \item Choose $m$ such that $\pi h m / l \in [2/M, M/2]$ is strictly in the real part of $E$.
\end{enumerate}
We now check all conditions of the theorem.
\begin{enumerate}
    \item $f(z) + g(z)$ fits the desired (eq \ref{prop N2 eq}), since $g(z) = O(h)$ after $z = O(1)$.
    \item $f(z)$ has one root $z = \pi h m /l - i\gamma_{12} h\log(1/h)$.
    \item $f(z)$ is clearly analytic as an affine function.
    \item $g(z)$ is analytic because the argument of $-C_1 C_2 / 4z^2$ admits a well-defined cut after making $\tilde{R}$ small. In particular, the largest term $z^2$ avoids the negative real axis since the closure of $E$ is slightly smaller than the lower right quadrant.
    \item $|f(z)| > |g(z)|$ by taking $h_0$ small enough. We elaborate by considering each side of the box $E$.
    \begin{align*}
        \Re z &= \f1M & 
        |f(z)| &\geq |\Re f(z)| = |\f{\pi h m}{l} - \Re z| \geq \f1M > |g(z)| \\
        \Re z &= M & 
        |f(z)| &\geq |\Re f(z)| = |\f{\pi h m}{l} - \Re z| \geq \f{M}2 > |g(z)| \\
        \Im z &= 0 &
        |f(z)| &\geq |\Im f(z)| = \gamma_{12} h \log(1/h) > |g(z)| \\
        \Im z &= -M &
        |f(z)| &\geq |\Im f(z)| \geq M - \f1M > |g(z)|
    \end{align*}
\end{enumerate}
Hence there exists a unique solution to the formula for Example \ref{ex N2 det}. A sequence of such $z$ can be found by choosing different $m$ in $\pi h m/l$. The real part of $E$ can be changed to any bounded interval away from 0.
\end{proof}

For greater $N$, we perform Rouche's theorem in pieces around every string of resonances separately. Furthermore, the equations have slightly different structure for the dominant string versus the remaining ones.

\begin{proposition} \label{prop existence dominant}
Consider any $N$ and the dominant string $\gamma_{JK}$. For any $M$ large, there exists $h_0$ small such that for all $h \in (0,h_0)$, any integer $m$ satisfying $\pi h m / (x_k-x_j) \in [\f{2}{M},\f{M}{2}]$ gives a resonance of the following form.
\begin{align} \label{eq prop dominant}
    z_m &= \f{\pi h m}{x_K-x_J} - i \gamma_{JK} h \log(1/h) + O(h)
\end{align}
\end{proposition}

\begin{proof}
We recall (eq \ref{lemma asymptotic}) and relegate what will be lower order terms to the function $\tilde{D}(z)$. Abbreviate $l = x_K - x_J$.
\begin{align}
    0 &= 1 + \f{C_J C_K}{4z^2} \* h^{\beta_J + \beta_K} \omega^{2(x_K - x_J)} + \tilde{D} \notag \\
    -e^{-2izl/h} &= \f{C_J C_K}{4z^2} \* h^{\beta_J + \beta_K} + \omega^{-2l} \tilde{D} \notag \\
    z &= \f{\pi h m}{l} - i \f{\beta_J+\beta_K}{2l} h \log(1/h) + \f{ih}{2l} \log \left( -\f{C_J C_K}{4z^2} - \f{1}{h^{\beta_J+\beta_K} w^{2l}} \tilde{D}\right) \label{z explicit level 1}
\end{align}
After taking the logarithm, the $\tilde{D}$ is more complicated than the $N=2$ case, but we can reduce it in the same way.
Our approach is to use Rouche's theorem on $f(z)$, $g(z)$, and $E \subset \mathbb{C}$ as follows.
\begin{align*}
    f(z) &:= \f{\pi h m}{l} - i\gamma_{JK} h\log(1/h) - z \\
    g(z) &:= \f{ih}{2l} \log \left( -\f{C_J C_K}{4z^2} - \f{1}{h^{\beta_J+\beta_K} w^{2l}} \tilde{D} \right) \\
    E &:= [1/M, M] + i[-(\gamma_{JK} + \eps) h\log(1/h)\, -(\gamma_{JK} - \eps) h\log(1/h)]
\end{align*}

\begin{figure}[H]
\centering
\begin{tikzpicture}
    [scale=0.8, ultra thick]
    \draw[fill=gray!15] (.2,-1) -- (5,-1) -- (5,-2) -- (.2,-2) --(.2,-1);
    \draw[->] (0.5,-1.5) -- (2.5,-1.5);
    \node[anchor=east] (g1) at (0, -2) {$-i(\gamma_{JK} + \eps) h\log(1/h)$};
    \node[anchor=east] (g1) at (0, -1) {$-i(\gamma_{JK} - \eps) h\log(1/h)$};
    \node[anchor=south] (g1) at (5, 0) {$M$};
    \node[anchor=south] (g1) at (0.2, 0) {$\f{1}{M}$};
    \node[] (g1) at (4, -1.5) {$E$};
    \draw[->, thin] (0,0) -- (5.5,0);
    \draw[->, thin] (0,0) -- (0,-3.5);
    \node[anchor=south] (axes) at (2.5,0) {$\Re z$};
    \node[anchor=south, rotate=90] (axes) at (0,-3) {$\Im z$};
\end{tikzpicture}
\end{figure}

We carefully describe the order of limits, narrowing the previous large box to a vertically skinny box. By Assumption \ref{assumption}, the only vertices on the slope of $\gamma_{JK}$ are $(0,0)$ and $\big(2(x_K-x_J,\beta_J+\beta_K\big)$, then Lemma \ref{lemma polygon} guarantees these two are strictly largest.
\begin{enumerate}
    \item Let $M >> 1$ fill the positive real axis with $[1/M, M]$.
    \item Choose $\eps$ such that $\tilde{D} = o(h^{\beta_J+\beta_K} \omega^{2l})$ throughout $E$. This is the $\eps$ in Lemma \ref{lemma polygon} such that $1$ and $\f{C_J C_K}{4z^2} \* h^{\beta_J + \beta_K} \omega^{2(x_K - x_J)}$ are the largest terms in $D_N$.
    \item Choose $h_0$ such that $|g(z)| < \eps h\log(1/h) < 1/M$ and $|\tilde{D} / (h^{\beta_J+\beta_K} \omega^{2l}) | << 1$ for all $h\in (0,h_0)$.
    \item Choose $m$ such that $\pi h m / l \in [2/M, M/2]$.
\end{enumerate}
Verifying the hypotheses of Rouch\'e's theorem follows similarly to Proposition \ref{prop N2}, but we elaborate on the inequalities in step $5$ for the new domain $E$. The closest inequality is $\eps h \log(1/h) > |g(z)|$, but it holds because $g(z) = O(h)$ and $\eps$ is independent of $h$.
\begin{align*}
    \Re z &= \f1M & 
    |f(z)| &\geq |\Re f(z)| = |\f{\pi h m}{l} - \Re z| \geq \f1M > |g(z)| \\
    \Re z &= M & 
    |f(z)| &\geq |\Re f(z)| = |\f{\pi h m}{l} - \Re z| \geq \f{M}2 > |g(z)| \\
    \Im z &= -(\gamma_{JK}-\eps) h\log(1/h) &
    |f(z)| &\geq |\Im f(z)| = \epsilon h \log(1/h) > |g(z)| \\
    \Im z &= -(\gamma_{JK}+\eps) h\log(1/h) &
    |f(z)| &\geq |\Im f(z)| = \epsilon h \log(1/h) > |g(z)|
\end{align*}
\end{proof}
Note that without Assumption \ref{assumption}, terms in $\tilde{D}$ might exceed $h^{\beta_J+\beta_K} \omega^{2l}$, and it seems like a different $f(z), g(z)$ would be necessary. In both Proposition \ref{prop existence dominant} and its subsequent analogue for other strings, controlling the largest terms is crucial.

\begin{proposition}
Consider any $N$ and the string $\tilde{\gamma}_{jk} \in \Gamma$. For any $M$ large, there exists $h_0$ small such that for all $h \in (0,h_0)$, any integer $m$ satisfying $\pi h m / (x_k-x_j) \in [\f{2}{M},\f{M}{2}]$ gives a resonance of the following form.
\begin{align} \label{eq prop subdominant}
    z_m &= \f{\pi h m}{x_k-x_j} - i \tilde{\gamma}_{jk} h \log(1/h) + O(h)
\end{align}
\end{proposition}
\begin{proof}
The previous proposition covers the case where $\gamma = \gamma_{JK}$. For the remaining strings, recall that the slope on the Newton polygon is between vertices which incorporate a third delta function, say $\beta_i$ at $x_i$. Proposition \ref{prop polygon intervals} allows us to take $x_i < x_j < x_k$ without loss of generality.
\begin{align*}
    \tilde{\gamma}_{jk} = \f{\beta_k-\beta_j}{2(x_k-x_j)} 
    = \f{(\beta_i+\beta_k)-(\beta_i+\beta_j)}{(x_k-x_i)-(x_j-x_i)}
\end{align*}
Using \ref{lemma polygon}, we write the equation for $z$ highlighting these largest terms, and relegate what will be lower order terms to the function $\tilde{D}(z)$. Again, Assumption \ref{assumption} guarantees that there are two strictly largest terms. Abbreviate $l = x_k - x_j$.
\begin{align}
    0 &= \f{C_i C_j}{4z^2} \* h^{\beta_i + \beta_j} \omega^{2(x_j - x_i)} + \f{C_i C_k}{4z^2} \* h^{\beta_i + \beta_k} \omega^{2(x_k - x_i)} + \tilde{D} \notag \\
    -e^{-2izl/h} &= \f{C_k}{C_j} \* h^{\beta_k-\beta_j} + \f{1}{h^{\beta_i+\beta_j} \omega^{2(x_k-x_i)}} \f{4z^2}{C_i C_j} \tilde{D} \notag \\
    z &= \f{\pi h m}{l} - i \f{\beta_k-\beta_j}{2l} h \log(1/h) + \f{ih}{2l} \log \left( -\f{C_k}{C_j} - \f{1}{h^{\beta_i+\beta_j} \omega^{2(x_k-x_i)}} \f{4z^2}{C_i C_j} \tilde{D}\right) \label{z explicit level 1 subdominant}
\end{align}
Our approach is to use Rouche's theorem on $f(z)$, $g(z)$, and $E \subset \mathbb{C}$ as follows.
\begin{align}
    f(z) &:= \f{\pi h m}{l} - i\gamma h\log(1/h) - z \\
    g(z) &:= \f{ih}{2l} \log \left( -\f{C_k}{C_j} - \f{1}{h^{\beta_i+\beta_j} \omega^{2(x_k-x_i)}} \f{4z^2}{C_i C_j} \tilde{D}\right) \\
    E &:= [1/2, 2] + i[-(\gamma_{JK}-\eps) h\log(1/h)\, -(\gamma_{JK}+\eps) h\log(1/h)]
\end{align}
The remaining steps are similar. For each $m$, there exists a unique solution to (eq \ref{z explicit level 1 subdominant}), and a longer sequence of $z$ can be found by taking $M$ large to admit more $m$ in $\pi h m / l \in [2/M, M/2]$.
\end{proof}
We have now proved the existence of strings of resonances obeying Theorem \ref{thm structure} and the uniqueness of those strings within that class. We have not proved that no other resonances exist, as opposed to Proposition \ref{prop N2}, but numerical evidence suggest they do not.

\section{Finishing Theorem \ref{thm existence} and Corollary \ref{cor log Re z}: refined asymptotics}

This section concerns the shape of the strings of resonances, where we refine the asymptotic estimates enough to match the numeric computations. Interestingly, we find that the string corresponding to the dominant pair of delta functions is logarithmic, while the remaining strings are flat. 
\begin{proof}[Proof of Theorem \ref{thm existence} and Corollary \ref{cor log Re z}]
We proceed in two cases, starting with the dominant string.
\begin{enumerate}[label=\alph*]
    \item % a
We start with the dominant string in (eq \ref{z explicit level 1}) and separate out $\tilde{D}$ as dominated by $h^{\beta_J+\beta_K} \omega^{2(x_K-x_J)}$.
\begin{align*}
    z &= \f{\pi h m}{l} - i \f{\beta_J+\beta_K}{2l} h \log(1/h) + \f{ih}{2l} \log \Big( -\f{C_J C_K}{4z^2}\Big) + o(h)
\end{align*}
Recall $l = x_K-x_J$ and that $m$ was chosen to make $z = O(1)$. We then plug the estimate $z = \pi h m /l + O\big(h\log(1/h)\big)$ into the above equation to refine the asymptotics.
\begin{align*}
    z &= \f{\pi h m}{l} - i \f{\beta_J+\beta_K}{2l} h \log(1/h) + \f{ih}{2l} \log \Big( -\f{C_J C_K}{4(\pi h m/l)^2 + O\big(h\log(1/h)\big)}\Big) + o(h) \\
    &= \f{\pi h m}{l} - i \f{\beta_J+\beta_K}{2l} h \log(1/h) + \f{ih}{2l} \log \left( -\f{C_J C_K}{4(\pi h m/l)^2} \* \sum_{n=0}^\infty \Big(-\f{O\big(h\log(1/h)\big)}{4 (\pi h m/l)^2} \Big)^n \right) + o(h) \\
    &= \f{\pi h m}{l} - i \f{\beta_J+\beta_K}{2l} h \log(1/h) + \f{ih}{2l} \left(\log \Big( -\f{C_J C_K}{4(\pi h m/l)^2} \Big) + O\big(h\log(1/h)\big) \right) + o(h) \\
    &= \f{\pi h m}{l} - i \f{\beta_J+\beta_K}{2l} h \log(1/h) + \f{ih}{2l} \log \Big( -\f{C_J C_K l^2}{4\pi^2 h^2 m^2} \Big) + o(h)
\end{align*}
Substituting for $l$ gives (eq \ref{thm existence eq dominant}). We say this equation is a log curve in the complex plane by approximating $\Re z = \pi h m /(x_K-x_J) + O(h)$ and computing the leading order of $\Im z$.
\begin{align*}
    \Im z &= - \f{\beta_J+\beta_K}{2(x_K-x_J)} h \log(1/h) + \f{h}{2(x_K-x_J)} \log \Big| \f{C_J C_K}{4(\Re z + O(h))^2} \Big| + o(h) \\
    &= - \f{\beta_J+\beta_K}{2(x_K-x_J)} h \log(1/h) - \f{h}{(x_K-x_J)} \log \Big(\f{2\Re z }{\sqrt{|C_J C_K|}}\Big) + o(h) \\
    &= - \f{h}{x_K-x_J} \log\left( \f{2 \Re z} {h^{(\beta_J+\beta_K)/2} \sqrt{|C_J C_K|}}\right) + o(h) 
\end{align*}
    \item % b
We perform a similar process for the non-dominant strings, though numerically, we expect the flat approximation to be sufficient. We start with (eq \ref{z explicit level 1 subdominant}) and separate $\tilde{D}$ as lower order.
\begin{align*}
    z &= \f{\pi h m}{l} - i \f{\beta_k-\beta_j}{2l} h \log(1/h) + \f{ih}{2l} \log \left( -\f{C_k}{C_j} + o(1) \right) \\
    &= \f{\pi h m}{(x_k-x_j)} - i \f{\beta_k-\beta_j}{2(x_k-x_j)} h \log(1/h) + \f{ih}{2(x_k-x_j)} \log \left( -\f{C_k}{C_j} \right) + o(h)
\end{align*}
This immediately gives (eq \ref{thm existence eq subdominant}), as well as that it is flat.
\begin{align*}
    \Im z &= - \f{\beta_k-\beta_j}{2(x_k-x_j)} h \log(1/h) + \f{h}{2(x_k-x_j)} \log \left| \f{C_k}{C_j} \right| + o(h)
\end{align*}
\end{enumerate}
\end{proof}

\section{Further numerical examples and code} \label{sec:numerics}

We use further numerical examples to fill several gaps in the theory we have developed so far, presented in order of most agreement with our theory to least. Then, we discuss how we compute the resonances. First, we provide a GitHub link to the code: \\
\href{https://github.com/Ethan-Brady/Explicit-Semiclassical-Resonances-from-Many-Delta-Functions}{https://github.com/Ethan-Brady/Explicit-Semiclassical-Resonances-from-Many-Delta-Functions}

\begin{example} \label{ex exact resonances}
While our graphical portrayal of theory has emphasized the curves in Corollary \ref{cor log Re z}, Theorem \ref{thm existence} actually gives the precise location of $\Re z$. This detail is of less interest to us but can be used to further evaluate the accuracy of the asymptotics. All cases of Example \ref{ex N4} match well, and we display three of them here.
\begin{figure}[H]
\centering
\includegraphics[width=0.3 \textwidth]{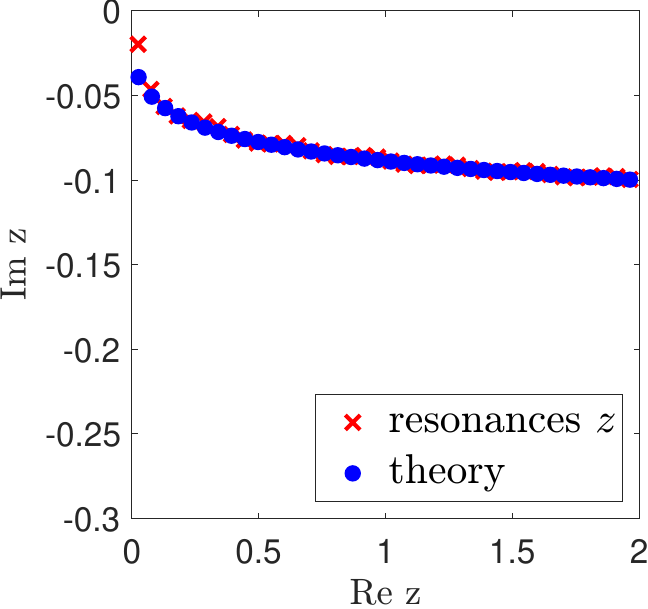}
\includegraphics[width=0.3 \textwidth]{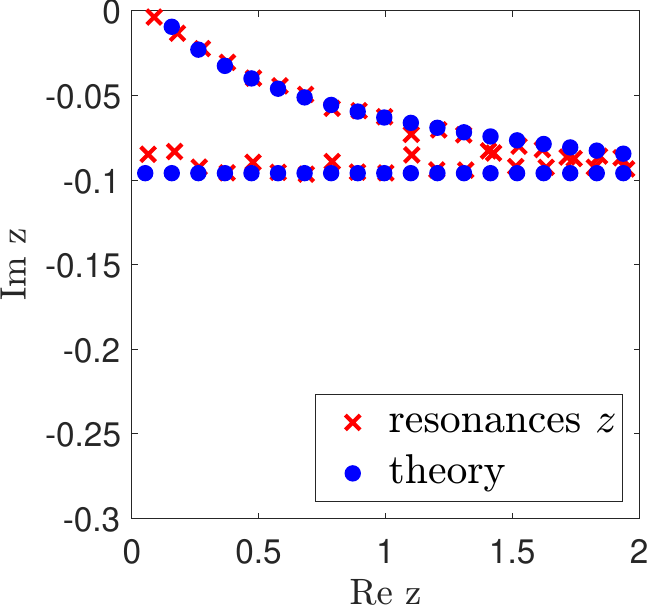}
\includegraphics[width=0.3 \textwidth]{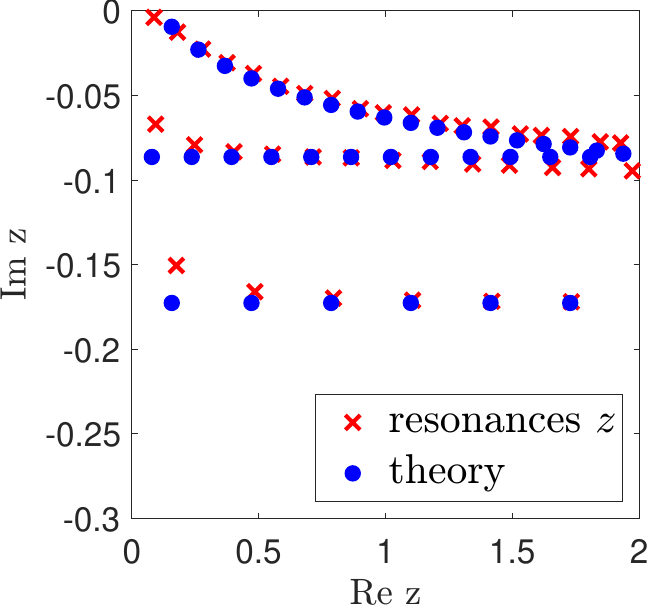}
\end{figure}
A related property to the specific location of $\Re z$ is the density of resonances, which appears to vary by string above. It might appear that lower strings beyond the dominant have fewer resonances, but the following case demonstrates this false. The density of resonances is determined by the term $\pi h m / (x_k-x_j)$ in Theorem \ref{thm existence}, that the number of resonances in a fixed interval is proportional to $x_k-x_j$ for that string. The inequalities of the Newton polygon favor smaller $x_k-x_j$ for lower strings beyond the dominant one, but choosing very strong $\beta_j$ can reverse this trend, as shown below.
\begin{figure}[H]
\centering
\resizebox{\textwidth}{!}{
\begin{tikzpicture} % deltas
[scale=0.6, ultra thick]
\draw[->] (0,0) -- (0,6);
\draw[->] (1,0) -- (1,6);
\draw[->] (3,0) -- (3,2);
\draw[->] (6,0) -- (6,.5);
\node[anchor=south] (b1) at (0.5,6) {$\beta_1 = \beta_2 = 0.05$};
\node[anchor=south] (b3) at (3,2) {$\beta_3 = 2.0$};
\node[anchor=south] (b4) at (6,.5) {$\beta_4 = 6.0$};
\node[anchor=north] (l1) at (4.5,0) {$x_2-x_1=3$};
\node[anchor=north] (l2) at (2,0) {$2$};
\node[anchor=north] (l3) at (0.5,0) {$1$};
\draw[<->, thin] (-0.5,0) -- (6.5,0);
\end{tikzpicture}
\begin{tikzpicture} % Newton polygon
    [scale=.6, ultra thick]
    \draw[fill=gray!15] (0,6.5) -- (0,0) -- (1,0.1) -- (3,2.05) -- (6,6.05) -- (6,6.5);
    \draw[->, thin] (0,0) -- (0,6);
    \draw[->, thin] (0,0) -- (6,0);
    \node[circle, fill, inner sep=1.5pt] (point) at (0,0) {};
    \node[circle, fill, inner sep=1.5pt] (point) at (1,0.1) {};
    \node[circle, fill, inner sep=1.5pt] (point) at (3,2.05) {};
    \node[circle, fill, inner sep=1.5pt] (point) at (2,2.05) {};
    \node[circle, fill, inner sep=1.5pt] (point) at (6,6.05) {};
    \node[circle, fill, inner sep=1.5pt] (point) at (5,6.05) {};
    \node[circle, fill, inner sep=1.5pt] (point) at (3,6.5) {};
    \node[anchor=north] (axes) at (3,0) {$\lambda = 2(x_k - x_j)$};
    \node[anchor=south, rotate=90] (axes) at (0,3) {$\nu = \beta_j + \beta_k$};
\end{tikzpicture}
\includegraphics[width=0.3 \textwidth]{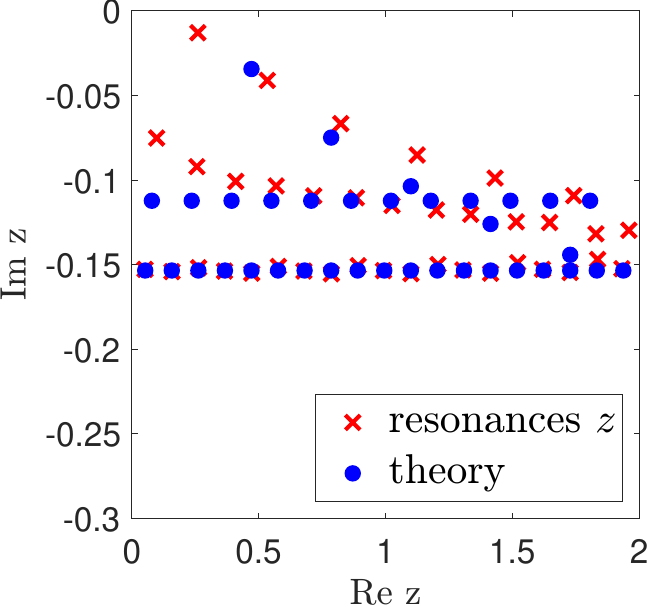}
} % resize box
\end{figure}
\end{example}

\begin{example} \label{ex assumption numeric}
    Without Assumption \ref{assumption}, we claim without proof that Theorem \ref{thm structure} and Theorem \ref{thm Gamma} remain true, but the interpretation becomes more complicated. In the following example, slopes on the Newton polygon do not correspond to connected intervals, since the second slope corresponds to $(x_1,x_2)$ and $(x_3,x_4)$ simultaneously. However, numerical computation reveals pairs of nearly equal resonances on this string, barely visible on the flat string, reminiscent of higher multiplicity roots. We implicitly defined strings of resonances by their value of $\gamma$, so a more formal description here would require a new definition of strings.
\begin{figure}[H]
\centering
\resizebox{\textwidth}{!}{
\begin{tikzpicture} % deltas
[scale=0.7, ultra thick]
    \draw[->] (1,0) -- (1,1);
    \draw[->] (2,0) -- (2,4);
    \draw[->] (5,0) -- (5,4);
    \draw[->] (6,0) -- (6,1);
    \node[anchor=south] (b1) at (1,1) {$\beta_1 = 2.0$};
    \node[anchor=south] (b2) at (2,4) {$\beta_2 = 0.5$};
    \node[anchor=south] (b3) at (5,4) {$\beta_3 = 0.5$};
    \node[anchor=south] (b4) at (6,1) {$\beta_4 = 2.0$};
    \node[anchor=south] (l1) at (1.5,0) {$1$};
    \node[anchor=south] (l2) at (3.5,0) {$x_3-x_2 = 3$};
    \node[anchor=south] (l3) at (5.5,0) {$1$};
    \draw[decorate, decoration = {brace, amplitude=10pt}] (5,-0.2) -- (2,-0.2);
    \draw[decorate, decoration = {brace, amplitude=3pt}] (2,-0.2) -- (1,-0.2);
    \draw[decorate, decoration = {brace, amplitude=3pt}] (6,-0.2) -- (5,-0.2);
    \draw[<->, thin] (-0.5,0) -- (6.5,0);
    \node (g) at (1,-0.8) {$\tilde{\gamma}_{12}$};
    \node (g) at (3.5,-1.2) {$\gamma_{23}$};
    \node (g) at (5.5,-0.8) {$\tilde{\gamma}_{34}$};
\end{tikzpicture}
\begin{tikzpicture} % Newton polygon
    [scale=.6, ultra thick]
    \draw[fill=gray!15] (0,6) -- (0,0) -- (3,1) -- (6,4) -- (6,6);
    \draw[->, thin] (0,0) -- (0,6);
    \draw[->, thin] (0,0) -- (6,0);
    \node[anchor=north west] (g) at (1.5,0.7) {$\gamma_{23}$};
    \node[anchor=north west] (g) at (4,2) {$\tilde{\gamma}_{12} = \tilde{\gamma}_{34}$};
    \node[circle, fill, inner sep=1.5pt] (point) at (0,0) {};
    \node[circle, fill, inner sep=1.5pt] (point) at (1,2.5) {};
    \node[circle, fill, inner sep=2.5pt] (point) at (1,2.5) {};
    \node[circle, fill, inner sep=1.5pt] (point) at (3,1.0) {};
    \node[circle, fill, inner sep=1.5pt] (point) at (4,2.5) {};
    \node[circle, fill, inner sep=2.5pt] (point) at (4,2.5) {};
    \node[circle, fill, inner sep=1.5pt] (point) at (6,4.0) {};
    \node[anchor=north] (axes) at (3,0) {$\lambda = 2(x_k - x_j)$};
    \node[anchor=south, rotate=90] (axes) at (0,3) {$\nu = \beta_j + \beta_k$};
\end{tikzpicture}
\includegraphics[width=0.3\textwidth]{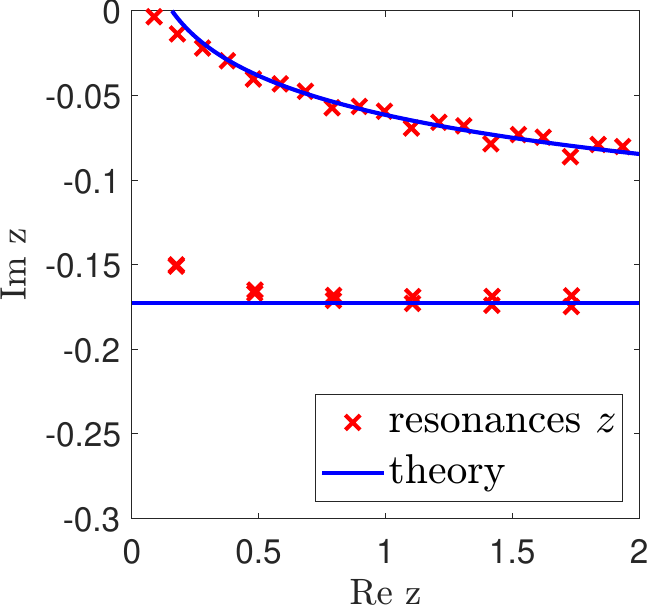}
} % resize box
\end{figure}
\end{example}

\begin{example} \label{ex intersection}
    We typically present examples where the values of $\gamma$ are not close to each other, since theory deviates from numerics when strings would intersect. The following shows the somewhat messy picture that results when they do intersect. This is inevitable for large $\Re x$, since the logarithmic curve of the first string will eventually intersect the flat ones, hence Conjecture \ref{conj large Re z} could fill this gap.
\begin{figure}[H]
\centering
\resizebox{\textwidth}{!}{
\begin{tikzpicture} % deltas
[scale=0.7, ultra thick]
    \draw[->] (0,0) -- (0,1.33);
    \draw[->] (2,0) -- (2,4);
    \draw[->] (5,0) -- (5,4);
    \draw[->] (6,0) -- (6,1);
    \node[anchor=south] (b1) at (0,1.33) {$\beta_1 = 1.5$};
    \node[anchor=south] (b2) at (2,4) {$\beta_2 = 0.5$};
    \node[anchor=south] (b3) at (5,4) {$\beta_3 = 0.5$};
    \node[anchor=south] (b4) at (6,1) {$\beta_4 = 2.0$};
    \node[anchor=south] (l1) at (1.0,0) {$2$};
    \node[anchor=south] (l2) at (3.5,0) {$x_3-x_2 = 3$};
    \node[anchor=south] (l3) at (5.5,0) {$1$};
    \draw[decorate, decoration = {brace, amplitude=10pt}] (5,-0.2) -- (2,-0.2);
    \draw[decorate, decoration = {brace, amplitude=3pt}] (2,-0.2) -- (0,-0.2);
    \draw[decorate, decoration = {brace, amplitude=3pt}] (6,-0.2) -- (5,-0.2);
    \draw[<->, thin] (-0.5,0) -- (6.5,0);
    \node (g) at (1,-0.8) {$\tilde{\gamma}_{12}$};
    \node (g) at (3.5,-1.2) {$\gamma_{23}$};
    \node (g) at (5.5,-0.8) {$\tilde{\gamma}_{34}$};
\end{tikzpicture}
\begin{tikzpicture} % Newton polygon
    [scale=.6, ultra thick]
    \draw[fill=gray!15] (0,6) -- (0,0) -- (3,1) -- (5,2.0) -- (6,3.5) -- (6,6);
    \draw[->, thin] (0,0) -- (0,6);
    \draw[->, thin] (0,0) -- (6,0);
    \node[anchor=north west] (g) at (1.5,0.7) {$\gamma_{23} = 1/3$};
    \node[anchor=north west] (g) at (4,1.75) {$\tilde{\gamma}_{12} = 1/2$};
    \node[anchor=north west] (g) at (5.5,3.25) {$\tilde{\gamma}_{34}$};
    \node[circle, fill, inner sep=1.5pt] (point) at (0,0) {};
    \node[circle, fill, inner sep=1.5pt] (point) at (1,2.5) {};
    \node[circle, fill, inner sep=1.5pt] (point) at (2,2.0) {};
    \node[circle, fill, inner sep=1.5pt] (point) at (3,1.0) {};
    \node[circle, fill, inner sep=1.5pt] (point) at (4,2.5) {};
    \node[circle, fill, inner sep=1.5pt] (point) at (5,2.0) {};
    \node[circle, fill, inner sep=1.5pt] (point) at (6,3.5) {};
    \node[anchor=north] (axes) at (3,0) {$\lambda = 2(x_k - x_j)$};
    \node[anchor=south, rotate=90] (axes) at (0,3) {$\nu = \beta_j + \beta_k$};
\end{tikzpicture}
\includegraphics[width=0.3\textwidth]{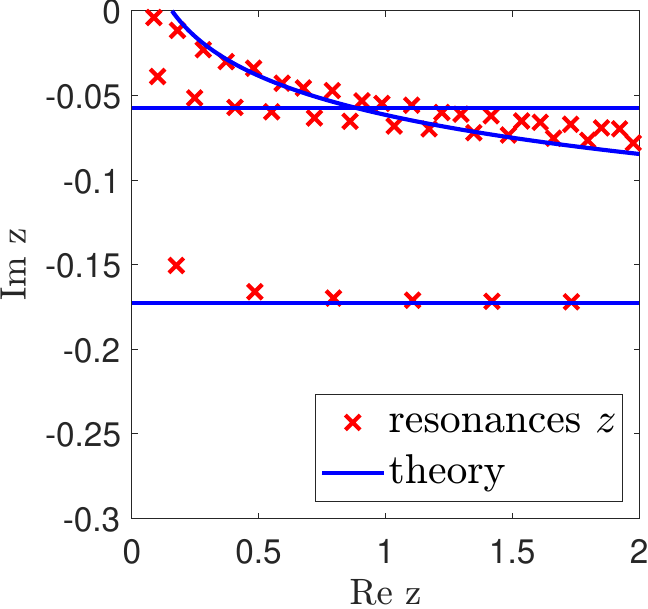}
} % resize box
\end{figure}
\end{example}

\begin{example} \label{ex C h}
    While the choice of $C_j \neq 0$ has no effect on the Newton polygon, it does affect the resonances numerically. Given the parameters on the left, the resonances in the middle graph have significant deviation from the theoretical approximations. Decreasing $h$ from $h=0.1$ as in the rest of the paper to $h=0.01$ in the graph on right improves the accuracy. Note that the theory here does not give error bounds more precise than the asymptotic $o(h)$.
\begin{figure}[H]
\centering
\resizebox{\textwidth}{!}{
\begin{tikzpicture} % deltas
[scale=0.7, ultra thick]
    \draw[->] (0,0) -- (0,1);
    \draw[->] (2,0) -- (2,4);
    \draw[->] (5,0) -- (5,4);
    \draw[->] (6,0) -- (6,1);
    \node[anchor=south, align=center] (b1) at (0,1) {$\beta_1 = 2.0$ \\ $C_1 = 10$};
    \node[anchor=south, align=center] (b2) at (2,4) {$\beta_2 = 0.5$ \\ $C_2 = 1$};
    \node[anchor=south, align=center] (b3) at (5,4) {$\beta_3 = 0.5$ \\ $C_3 = -5$};
    \node[anchor=south, align=center] (b4) at (6,1) {$\beta_4 = 2.0$ \\ $C_4 = 1$};
    \node[anchor=south] (l1) at (1.0,0) {$2$};
    \node[anchor=south] (l2) at (3.5,0) {$x_3-x_2 = 3$};
    \node[anchor=south] (l3) at (5.5,0) {$1$};
    \draw[decorate, decoration = {brace, amplitude=10pt}] (5,-0.2) -- (2,-0.2);
    \draw[decorate, decoration = {brace, amplitude=3pt}] (2,-0.2) -- (0,-0.2);
    \draw[decorate, decoration = {brace, amplitude=3pt}] (6,-0.2) -- (5,-0.2);
    \draw[<->, thin] (-0.5,0) -- (6.5,0);
    \node (g) at (1,-0.8) {$\tilde{\gamma}_{12}$};
    \node (g) at (3.5,-1.2) {$\gamma_{23}$};
    \node (g) at (5.5,-0.8) {$\tilde{\gamma}_{34}$};
\end{tikzpicture}
\includegraphics[width=0.3\textwidth]{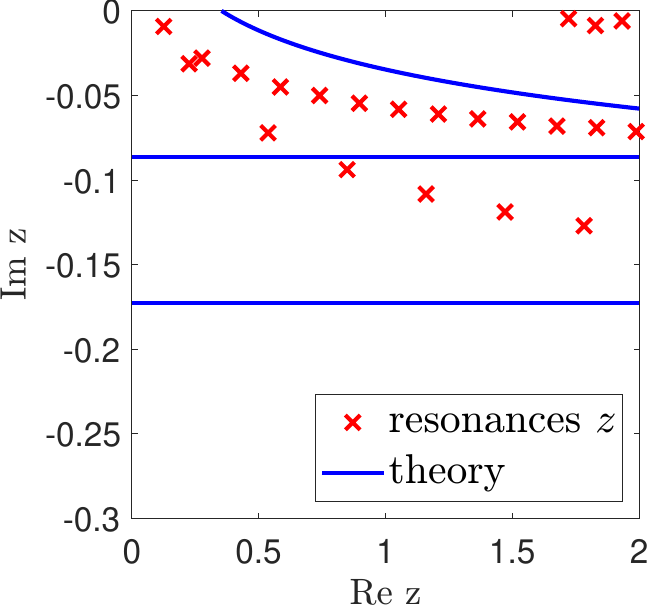}
\includegraphics[width=0.3\textwidth]{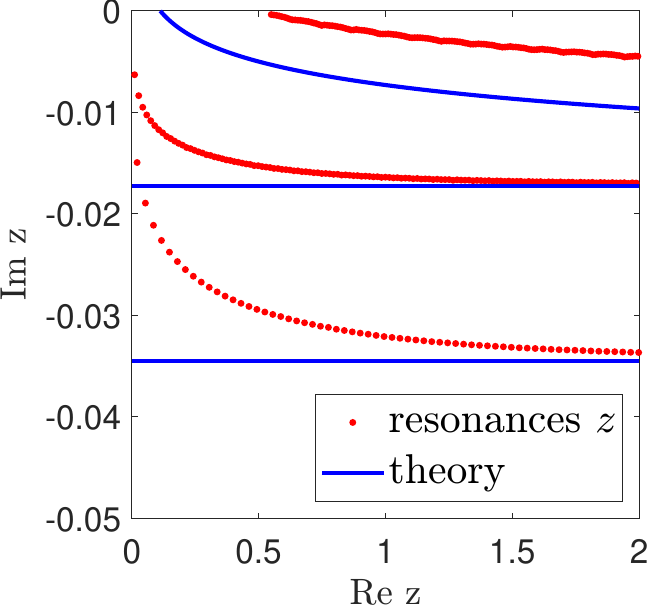} 
} % resize box
\end{figure}
\end{example}

Our method of computing resonances, inspired by \cite{Datchev_Malawo} and similar to \cite{Erman_et_al_2018}, is depicted in Figure \ref{fig numerics}. After manual or symbolic calculation to find the determinant (eq \ref{lemma det eq}), we have the common task of finding its roots. However, naive gradient descent did not work for small $h$, so we compute the zero level contours of the real and imaginary part of $D_N$ separately and find their intersections. The former task is standard with MATLAB's fcontour(), while we require a custom function (Fast and Robust Curve Intersections by Douglas Schwarz \cite{intersections}) to compute the possibly large number of intersections. A graphic of this computation and the full code follows. \par
Lastly, we numerically found no resonances outside of the $h \log(1/h)$ regime discussed. With the visual confirmation from the zero level contours, we expect this computation to be reliable in the given window and less so as $\Re z\rightarrow \infty$ or $\Im z \rightarrow -\infty$. This latter range was initially challenging, but the zero level contours for the imaginary part do not extend too far negative in a variety of examples.

\begin{figure}[H]
\centering
\includegraphics[width=0.4\textwidth]{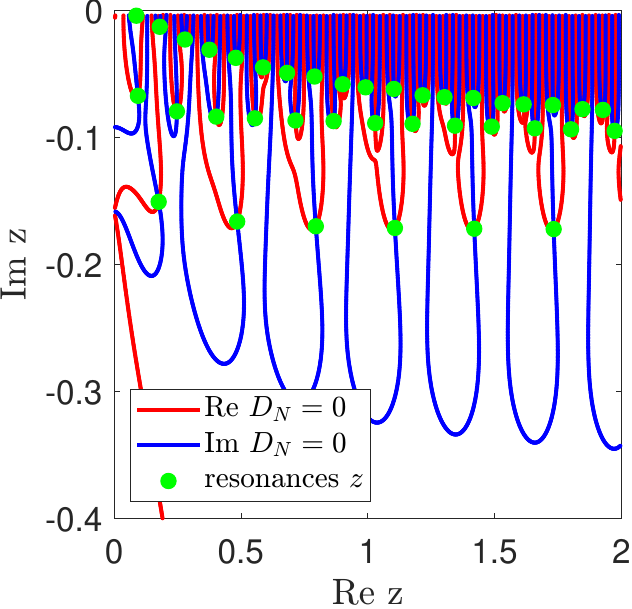}
\includegraphics[width=0.42\textwidth]{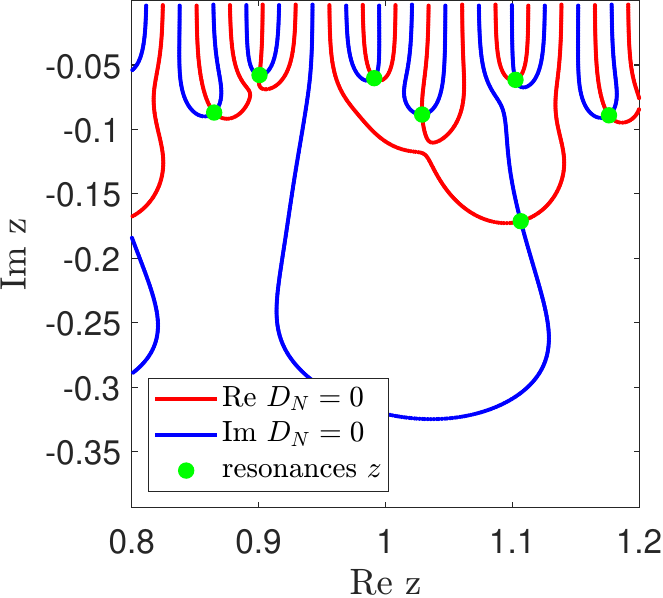}
\caption{ Internal method of numerically computing the resonances. Red curves are 0 level contours for the real part of equation $D_N = 0$ (eq \ref{lemma det eq}) governing resonances $z$, blue curves are the imaginary part, and green dots are their intersections. The right graph is a zoomed-in version of the left, which corresponds to Example \ref{ex N4} case $c$.}
\label{fig numerics}
\end{figure}

\section{Acknowledgements}

I would like to deeply thank Professor Kiril Datchev for his ample guidance, introduction to the problem, and thorough edits. I would also like to thank Professors Jared Wunsch and Jeremy Marzuola for insightful discussions. This research was conducted at Purdue University.

\printbibliography[heading=bibintoc]

\end{document}